\documentclass[11pt,oneside,reqno]{amsart}
\usepackage{geometry}
\usepackage[utf8]{inputenc}
\usepackage{amstext,latexsym,amsbsy,amsmath,amssymb,amsthm,mathtools,relsize,geometry,enumerate}
\usepackage{hyperref}
\hypersetup{
  colorlinks   = true, %Colours links instead of ugly boxes
  urlcolor     = blue, %Colour for external hyperlinks
  linkcolor    = red, %Colour of internal links
  citecolor   = red %Colour of citations
}
\usepackage{mathtools,enumerate,enumitem}

\usepackage{graphicx}
\usepackage{epstopdf}
\setkeys{Gin}{width=\linewidth,totalheight=\textheight,keepaspectratio}
\graphicspath{{./images/}}

\usepackage{multicol}
\usepackage{booktabs}
\usepackage{mathrsfs}
\usepackage{color}

\newcommand{\nbb}{\mathbb{N}}
\newcommand{\rbb}{\mathbb{R}}

\newcommand{\W}{\mathcal{W}}
\renewcommand{\H}{\mathcal{H}}

\newcommand{\B}{\mathcal{B}}
\newcommand{\la}{\langle}
\newcommand{\ra}{\rangle}

\newcommand{\x}{\mathrm{x}}
\newcommand{\y}{\mathrm{y}}
\newcommand{\Xtil}{\widetilde{X}}

\newcommand{\uhat}{\widehat{u}}
\newcommand{\etahat}{\widehat{\eta}}
\newcommand{\ubar}{\overline{u}}
\newcommand{\etabar}{\overline{\eta}}

\newcommand{\utilde}{\widetilde{u}}

\newcommand{\etatilde}{\widetilde{\eta}}
\newcommand{\Phitilde}{\widetilde{\Phi}}
\newcommand{\wtilde}{\widetilde{w}}

\newcommand{\mi}{\wedge}

\newcommand{\Tr}{\text{Tr}}

\renewcommand{\d}{\text{d}}

\newcommand{\domain}{\mathcal{O}}

\newcommand{\f}{\varphi}

\newcommand{\grad}{\nabla}
\newcommand{\Dom}{\text{Dom}}

\newcommand{\Fcal}{\mathcal{F}}
\newcommand{\E}{\mathbb{E}}

\renewcommand{\P}{\mathbb{P}}

\newcommand{\Law}{\text{Law}}

\newcommand{\A}{\mathcal{A}}
\newcommand{\M}{\mathcal{M}}

\newcommand{\nbar}{\bar{n}}

\newcommand{\close}{\!\!\!}

\newcommand{\TV}{\text{TV}}

\theoremstyle{plain}
\newtheorem{theorem}{Theorem}[section]

\newtheorem{lemma}[theorem]{Lemma}
\newtheorem{assumption}[theorem]{Assumption}
\newtheorem{proposition}[theorem]{Proposition}
\newtheorem{definition}[theorem]{Definition}

\newtheorem{remark}[theorem]{Remark}

\numberwithin{equation}{section}

\title{Ergodicity of a nonlinear stochastic reaction-diffusion equation with memory}

\author{ Hung D.~Nguyen$^1$}

\address{$^1$ Department of Mathematics, University of California, Los Angeles, California, USA}

\begin{document}

\begin{abstract}
We consider a class of semi-linear differential Volterra equations with memory terms, polynomial nonlinearities and random perturbation. For a broad class of nonlinearities, we show that the system in concern admits a unique weak solution. Also, any statistically steady state must possess regularity compatible with that of the solution. Moreover, if sufficiently many directions are stochastically forced, we employ the \emph{generalized coupling} approach to prove that there exists a unique invariant probability measure and that the system is exponentially attractive. This extends ergodicity results previously established in [Bonaccorsi et al., SIAM J. Math. Anal., 44 (2012)]. 
\end{abstract}
\maketitle

\section{Introduction} \label{sec:intro}
\subsection{Overview} \label{sec:intro:overview}
Let $\domain\in \rbb^d$ be a bounded open domain with smooth boundary. We consider the following system in the unknown variable $u(t)=u(x,t):\domain\times\rbb\to\rbb$
\begin{equation} \label{eqn:react-diff:K:original}
\begin{aligned}
\d u(t)&=\triangle u(t)\d t -\int_{0}^{\infty}\close K(s)\triangle u(t-s)\d s\d t+\f(u(t))\d t+Q\d w(t),\\
u(t)\big|_{\partial\domain}&=0,\qquad u(s)=u_0(s),\,\,  s\le 0,
\end{aligned}
\end{equation}
modeling the temperature of a heat flow by conduction in viscoelastic materials \cite{ coleman1967equipresence,coleman1964material,gurtin1968general,nunziato1971heat}. In~\eqref{eqn:react-diff:K:original}, $K:[0,\infty)\to[0,\infty)$ is the heat flux memory kernel satisfying $\int_0^\infty K(s)\d s<1$, $\f$ is a nonlinear term with polynomial growth, $w(t)$ is a cylindrical Wiener process, and $Q$ is a linear bounded map on certain Hilbert spaces. For simplicity, we set all physical constants to 1.

In the absence of memory effects, that is when $K \equiv 0$, we note that~\eqref{eqn:react-diff:K:original} is reduced to the classical stochastic reaction-diffusion equation
\begin{equation} \label{eqn:react-diff}
\begin{aligned}
\d u(t)&=\triangle u(t)\d t +\f(u(t))\d t+Q\d w(t),\\
u(t)\big|_{\partial\domain}&=0,\qquad u(0)=u_0.
\end{aligned}
\end{equation}
Under different assumptions on $\f$ as well as the noise term, statistically steady states of~\eqref{eqn:react-diff} are well studied. Using a Krylov-Bogoliubov argument together with tightness,~\eqref{eqn:react-diff} always admits at least one invariant probability measure \cite{da2014stochastic}. For both Lipschitz and dissipative assumptions on $\f$, it is a classical result that such an invariant measure is unique for~\eqref{eqn:react-diff} \cite{cerrai2001second,da1996ergodicity,da2014stochastic}. Under more general conditions on $\f$, provided that noise is sufficiently forced in many directions of the phase space,~\eqref{eqn:react-diff} satisfies the so-called \emph{asymptotic strong Feller} property and hence exponentially attractive toward equilibrium \cite{hairer2011theory}. 

On the other hand, in the presence of memory kernels without random perturbation, i.e., $Q\equiv 0$, the well-posedness of the following equation
\begin{equation} \label{eqn:react-diff:K:deterministic}
\begin{aligned}
\d u(t)&=\triangle u(t)\d t -\int_{0}^{\infty}\close K(s)\triangle u(t-s)\d s\d t+\f(u(t))\d t,\\
u(t)\big|_{\partial\domain}&=0,\qquad u(s)=u_0(s),\,\,  s\le 0,
\end{aligned}
\end{equation}
was studied as early as in the work of \cite{
miller1974linear}. Moreover, positivity of the solutions for~\eqref{eqn:react-diff:K:deterministic} as well as large-time asymptotic results were established in \cite{clement1981asymptotic2,clement1981asymptotic}. The topic of global attractors was explored for a variety of deterministic systems related to~\eqref{eqn:react-diff:K:deterministic} \cite{chekroun2010asymptotics,chekroun2012invariant,
conti2005singular,conti2006singular,pata2001attractors}. Equation~\eqref{eqn:react-diff:K:deterministic} was also the motivation for many works of differential Volterra equations in the context of memory \cite{barbu1975nonlinear,barbu52nonlinear,barbu1979existence,barbu2010nonlinear}.

Although there is a rich literature on~\eqref{eqn:react-diff} and \eqref{eqn:react-diff:K:deterministic}, much less is known about asymptotic behaviors of the stochastic equation~\eqref{eqn:react-diff:K:original}. To the best of the author's knowledge, first result in this direction seems to be established in the work of~\cite{bonaccorsi2012asymptotic}. For a wide class of nonlinearities, mild solutions of~\eqref{eqn:react-diff:K:original} was constructed \cite[Section 4.2]{bonaccorsi2012asymptotic} via the classical Yosida approximation. Moreover, under the assumption of either $\f\equiv0$ or dissipative nonlinearity, that is $\f'\le 0$, it was shown that~\eqref{eqn:react-diff:K:original} admits a unique invariant probability measure and that the system is mixing \cite[Theorem 5.1]{bonaccorsi2012asymptotic}. Similarly method was also employed in \cite{caraballo2007existence,caraballo2008pullback} to prove the existence of random attractors. The goal of this note is to make further progress toward statistically steady states of~\eqref{eqn:react-diff:K:original} under more general conditions on $\f$ and the stochastic forcing.

In general, due to the memory effect, $u(t)$ as in~\eqref{eqn:react-diff:K:original} is not a Markov process. It is thus convenient to transform~\eqref{eqn:react-diff:K:original} to a Cauchy problem on suitable spaces accounting for the whole past information. To see this, following the framework in~\cite{bonaccorsi2012asymptotic,miller1974linear}, we introduce the auxiliary \emph{memory} variable $\eta$ given by
\begin{align*}
\eta(t;s)= u(t-s),
\end{align*}
and observe that formally $\eta$ satisfies the following transport equation
\begin{align} \label{eqn:eta:original}
\partial_t\eta=-\partial_s\eta,\quad \eta(t;0)=u(t).
\end{align}
So that the above equation together with~\eqref{eqn:react-diff:K:original} forms a Markovian dynamics in the pair $(u(t),\eta(t))$ (see~\eqref{eqn:react-diff:K} below). It is thus necessary to construct memory spaces in which the dynamic $\eta(t)$ evolves over time $t$ (see Section~\ref{sec:functional-setting} below). The method of augmenting~\eqref{eqn:react-diff:K:original} by memory variables and spaces was employed as early as in the work of~\cite{miller1974linear} and later popularized in~\cite{bonaccorsi2012asymptotic}. One of the main difficulties dealing with these memory spaces, however, is the lack of compact embeddings that are typically found in classical Sobolev spaces. Therefore, when it comes to studying large-time asymptotics, it is a challenging problem to prove the existence of invariant measures via the Krylov-Bogoliubov procedure. To circumvent the issue, in~\cite{bonaccorsi2012asymptotic}, based on the assumption that $\f'\le 0$, it was shown that the invariant measure uniquely exists. Because of such dissipative condition, however, the method employed therein does not cover more general potential instances, e.g., the Allen-Cahn function $\f(x)=x-x^3$. On the other hand, in the absence of memory, it is well-known that one can employ the so-called \emph{generalized coupling} argument~\cite{butkovsky2020generalized,hairer2011theory,
kulik2015generalized,kulik2017ergodic} to study invariant structures of~\eqref{eqn:react-diff} with the Allen-Cahn function. One of our goal here therefore is to generalize the ergodic results previously established in~\cite{bonaccorsi2012asymptotic} making use of the framework from~\cite{butkovsky2020generalized,hairer2011theory,
kulik2015generalized,kulik2017ergodic} applied to our settings in the context of memory. 

Unlike the mild solution approach of~\cite{bonaccorsi2012asymptotic,caraballo2007existence,
caraballo2008pullback}, we will study the well-posedness of~\eqref{eqn:react-diff:K:original} in a weak formulation. More specifically, under a general condition on the nonlinearities, we will construct the solution via a Galerkin approximation. As a byproduct of doing so, we are able to obtain useful moment bounds on the solutions. In turn, this will establish our first main result about the regularity of statistically steady states that is compatible to that of the weak solutions, cf. Theorem~\ref{thm:mu:moment-bound}. In the second main result concerning unique ergodicity, under a slightly stronger assumption on the random perturbation, namely, noise is sufficiently forced in many Fourier directions, we employ the approach of \emph{generalized coupling} developed in~\cite{butkovsky2020generalized,hairer2011theory} to prove the existence and uniqueness of the invariant probability measure $\mu$ to which the system \eqref{eqn:react-diff:K:original} is exponentially attracted in suitable Wasserstein distances, cf. Theorem~\ref{thm:ergodicity}.

\newcommand{\ebold}{\mathrm{e}}
\newcommand{\qbold}{\mathrm{q}}

\subsection{Physical motivation}\label{sec:intro:motivation}

Equation~\eqref{eqn:react-diff:K:original} is derived from the heat conduction in a viscoelastic material such as polymer \cite{coleman1964material,nunziato1971heat}. More specifically, by Fick's law of balance of heat, it holds that \cite{bonaccorsi2012asymptotic}
\begin{align*}
\partial_t\, \ebold +\text{div}\, \qbold =r,
\end{align*}
where $\ebold$ represents the internal energy, $r$ is an external heat supply and $\qbold$ is the heat flux of the form
\begin{align*}
\qbold(t,x) = -\int_0^\infty  \close \d a(s)\grad u(t-s,x).
\end{align*}
In Fourier conductors, $a\equiv k_0$, where $k_0$ is the instantaneous conductivity constant. So that, $\qbold$ is reduced to
\begin{align*}
\qbold(t,x) = -k_0\grad u (t,x),
\end{align*}
which yields the classical nonlinear heat equation~\eqref{eqn:react-diff} (by setting $k_0=1$). This amounts to the assumption that there is no time correlation between the heat transfer and the surrounding medium. However, as pointed out elsewhere~ \cite{coleman1964material,coleman1967equipresence,
gurtin1968general}, the above type of $\qbold$ does not account for memory effects in viscoelastic materials, particularly at low temperature \cite{nunziato1971heat}. It is thus more appropriate to consider $a(t)$ given by
\begin{align*}
a(t)=k_0\pm\int_0^t k(s)\d s,
\end{align*}
where $k\in[0,\infty)\to(0,\infty)$ is the memory kernel that is decreasing on $[0,\infty)$. With regard to the sign of the memory integral, there have been many works considering $a(t)=k_0+\int_0^t k(s)\d s$, as $\int_0^t k(s)\d s$ representing the thermal drag in addition to instantaneous conductivity. See \cite{caraballo2007existence,caraballo2008pullback,
coleman1964material,conti2005singular,coleman1967equipresence,
conti2006singular,gurtin1968general,nunziato1971heat} and the references therein. On the other hand,~\cite{bonaccorsi2012asymptotic,
clement1981asymptotic2,clement1981asymptotic} as well as this note study $a(t)=k_0-\int_0^t k(s)\d s$, which yields~\eqref{eqn:react-diff:K:original} by setting $k_0=1$. It is worth to mention that for physical reasons, one should require \cite{clement1981asymptotic2}
\begin{align*}
k_0=1>\int_0^\infty\close k(s)\d s.
\end{align*}
Later in Sections~\ref{sec:moment-bound:solution}-\ref{sec:ergodicity}, we will see that the above condition is actually critical in the analysis of~\eqref{eqn:react-diff:K:original}. 

%We remark that there are other physical phenomena that can also be modeled using~\eqref{eqn:react-diff:K:original}, e.g., population dynamics \cite{cantrell2004spatial}. 

Finally, we remark that in literature, stochastic equations with memory was studied as early as in the seminal work of~\cite{ito1964stationary} for finite-dimensional settings. The theory of stationary solutions was then developed further in \cite{bakhtin2005stationary} and was employed to study a variety of finite-dimensional and infinite-dimensional settings, e.g., Navier-Stokes equation \cite{weinan2001gibbsian}, Ginzburg-Landau equation \cite{weinan2002gibbsian} and Langevin equation \cite{herzog2021gibbsian}. In particular, heat equation with memory was studied in a series of work in \cite{barbu1975nonlinear,barbu52nonlinear,barbu1979existence,
barbu2010nonlinear}

The rest of the paper is organized as follows: in Section~\ref{sec:functional-setting}, we introduce all the functional settings needed for the analysis. In particular, we will see that~\eqref{eqn:react-diff:K:original} induces a Markovian dynamics as an abstract Cauchy equation evolving on an appropriate product space. In Section~\ref{sec:result}, we introduce the main assumptions on the non-linearities and noise structure. We also state our main results in this section, including Theorem~\ref{thm:mu:moment-bound} on the regularity of invariant measures and Theorem~\ref{thm:ergodicity} on geometric ergodicity. In Section~\ref{sec:moment-bound:solution}, we collect a priori moment bounds on the solutions that will be employed to prove the main results. In Section~\ref{sec:moment-bound:mu}, we establish regularity of an invariant measure. We then discuss the coupling approach and prove geometric ergodicity in Section~\ref{sec:ergodicity}. In Appendix~\ref{sec:well-posed}, we employ Galerkin approximation to construct the solutions of~\eqref{eqn:react-diff:K:original}.

\section{Functional setting} \label{sec:functional-setting}
Letting $\domain$ be a smooth bounded domain in $\rbb^d$, we denote by $H$ the Hilbert space $L^2(\domain)$ endowed with the inner product $\la\cdot,\cdot\ra_H$ and the induced norm $\|\cdot\|_H$.

 Let $A$ be the realization of $-\triangle$ in $H$ endowed with the Dirichlet boundary condition and the domain $\Dom(A)=H^1_0(\domain)\cap H^2(\domain)$. It is well-known that there exists an orthonormal basis $\{e_k\}_{k\ge 1}$ in $H$ that diagonalizes $A$, i.e.,  
\begin{equation}\label{eqn:Ae_k=alpha_k.e_k}
Ae_k=\alpha_k e_k,
\end{equation}
for a sequence of positive numbers $\alpha_1<\alpha_2<\dots$ diverging to infinity. 

For each $r\in\rbb$, we denote
\begin{equation}
H^r=\Dom(A^{r/2}),
\end{equation}
endowed with the inner product
\begin{align*}
\la u_1,u_2\ra_{H^r}=\la A^{r/2}u_1,A^{r/2}u_2\ra_H.
\end{align*}
In view of~\eqref{eqn:Ae_k=alpha_k.e_k}, the inner product in $H^r$ may be rewritten as \cite{cerrai2020convergence,conti2005singular,conti2006singular}
\begin{align*}
\la u_1,u_2\ra_{H^r}=\sum_{k\ge 1}\alpha_k^{r}\la u_1,e_k\ra_H\la u_2,e_k\ra_H.
\end{align*}
The induced norm in $H^r$ then is given by
\begin{align*}
\|u\|^2_{H^r}=\sum_{k\ge 1}\alpha_k^{r}|\la u_1,e_k\ra_H|^2.
\end{align*}
It is well-known that the embedding $H^{r_1}\subset H^{r_2}$ is compact for $r_1>r_2$. For $n\ge 1$, we denote by $P_n$ the projection onto the first $n$ wavenumbers $\{e_1,\dots,e_n\}$, i.e.,
\begin{equation} \label{form:P_n.u}
P_nu=\sum_{k=1}^n \la u,e_k\ra_He_k.
\end{equation}
The above projection will be useful in Sections~\ref{sec:moment-bound:mu} and \ref{sec:ergodicity} when we study the asymptotic behavior of~\eqref{eqn:react-diff:K:original}.

In order to treat the memory term of~\eqref{eqn:react-diff:K:original} on an extended phase space, following the framework in~\cite{bonaccorsi2012asymptotic}, we introduce the function $\rho(t)$ given by
\begin{equation} \label{form:rho}
\rho(t) =\int_t^\infty\close K(s)\d s,
\end{equation}
and the weighted space
\begin{equation}\label{form:M}
\M^r:=L^2_\rho(\rbb^+;H^{r+1}),\quad r\in\rbb,
\end{equation}
endowed with the inner product
\begin{equation} \label{form:M:norm}
\la\eta_1,\eta_2\ra_{\M^r}=\int_0^\infty\close \rho(s)\la \eta_1(s),\eta_2(s)\ra_{H^{r+1}}\d s.
\end{equation} 
As mentioned in Section~\ref{sec:intro:overview}, unlike the compact embedding $H^{r_1}\subset H^{r_2}$, $r_1>r_2$, the embedding $\M^{r_1}\subset \M^{r_2}$ is only continuous \cite{conti2005singular,conti2006singular}. 

Next, we define $\H^r$ to be the product space given by
\begin{equation} \label{form:HxM}
\H^r=H^r\times\M^r,\quad r\in\rbb,
\end{equation}
endowed with the norm
\begin{equation*}
\|\x\|^2_{\H^r} = \|u\|^2_{H^r}+\|\eta\|^2_{\M^r}.
\end{equation*}
To simplify notation, we shall use $\M$ and $\H$ instead of $\M^0$ and $\H^0=H\times\M^0$, respectively. For $\x=(u,\eta)\in\H^r$, the projections of $\x$ onto the marginal spaces are given by
\begin{equation}
\pi_1\x=u\in H^r\quad\text{and}\quad\pi_2\x=\eta\in \M^r.
\end{equation}
On the Hilbert space $\H$, we consider the operator $\A$ defined for $\x=(u,\eta)$
\begin{equation} \label{form:A}
\A\x=\begin{pmatrix}
-Au+\int_0^\infty K(s)A\eta(s)\d s\\
-\partial_s
\end{pmatrix},
\end{equation}
with the domain \cite[Section 2.3]{bonaccorsi2012asymptotic}
\begin{align*}
\Dom(\A) = \{\x=(u,\eta)\in\H| u\in H^2,\eta\in W^{1,2}_\rho(\rbb^+;H^2), \eta(0)=u  \}.
\end{align*}
It can be shown that $\A$ generates a strong continuous semigroup of contractions $S(t)$ in $\H$ \cite[Theorem 2.5]{bonaccorsi2012asymptotic}. Concerning the operator $\partial_s$,
by the choice of $\rho$ as in~\eqref{form:rho}, observe that
\begin{align}\label{eqn:<partial_s.eta,eta>}
\la -\partial_s\eta,\eta\ra_\M& =-\int_0^\infty\close \rho(s)\tfrac{1}{2}\partial_s\|\eta(s)\|^2_{H^1_0}\d s
=\tfrac{1}{2}\rho(0)\|\eta(0)\|^2_{H^1}+\tfrac{1}{2}\int_0^\infty\close\rho'(s)\|\eta(s)\|^2_{H^1}\d s . 
\end{align}
We will make use of the above identity later in the analysis of~\eqref{eqn:react-diff:K:original}. Furthermore, given $u\in L^2(0,T;H^1)$, the following transport equation
\begin{equation}\label{eqn:partial_s}
\tfrac{\d}{\d t}\eta(t;\cdot)=-\partial_s\eta(t;\cdot),\quad \eta(0;\cdot)=\eta_0\in \M,\quad \eta(t;0)=u(t),
\end{equation}
admits a unique solution given by (see \cite[Expression (2.17)]{bonaccorsi2012asymptotic} and \cite[Proposition 1.2]{pruss2013evolutionary})
\begin{equation} \label{form:eta}
\eta(t;s)=u(t-s)\boldsymbol{1}\{s\le t\}+\eta_0(s-t)\boldsymbol{1}\{s> t\}.
\end{equation}
Also, for $\etatilde\in \M$ such that $\partial\etatilde\in \M$, by integration by parts, we note that \cite{bonaccorsi2012asymptotic}
\begin{align}
&\tfrac{\d}{\d t}\la \eta(t),\etatilde\ra_\M =\la -\partial_s\eta(t),\etatilde\ra_\M\nonumber\\
&=-\rho(s)\la \eta(t;s),\etatilde(s)\ra_{H^1}\big|^\infty_0+\int_0^\infty\close \la \eta(t;s),\partial_s\big(\rho(s)\etatilde(s)\big)\ra_{H^1}\d s\nonumber\\
&=\rho(0)\la u(t),\etatilde(0)\ra_{H^1} +\int_0^\infty\close\rho'(s)\la \eta(t;s),\etatilde(s)\ra_{H^1}\d s+ \int_0^\infty\close\rho(s)\la \eta(t;s),\partial_s\etatilde(s)\ra_{H^1}\d s .
\label{eqn:eta:int-by-parts}
\end{align}

Having introduced suitable spaces, we may transform~\eqref{eqn:react-diff:K:original} to the following equation
\begin{equation} \label{eqn:react-diff:K}
\begin{aligned}
\d\begin{pmatrix}
u(t)\\\eta(t)
\end{pmatrix}=\A\begin{pmatrix}
u(t)\\\eta(t)
\end{pmatrix}\d t+\begin{pmatrix}
\f(u(t))\\0
\end{pmatrix}\d t+\begin{pmatrix}
Q \\ 0
\end{pmatrix}\d w(t),
\end{aligned}
\end{equation}
where $\A$ is as in~\eqref{form:A} together with the conditions
\begin{align*}
u(0)=u_0\in H,\eta(0)=\eta_0\in \M,\, \eta(t;0)=u(t),\, t>0.
\end{align*}

\section{Main results}  \label{sec:result}

\subsection{Well-posedness} \label{sec:result:subsection}
We begin this section by stating the following condition on the kernel $K$:
\begin{assumption} \label{cond:K} The kernel $K(t):[0,\infty)\to (0,\infty)$ satisfies
\begin{equation} \label{cond:K:1}
\int_0^\infty\close K(s)\emph{d} s < 1,
\end{equation}
and
\begin{align} \label{cond:K:2}
K'(s)+\delta K(s)\le 0, \quad s> 0,
\end{align}
for some constant $\delta>0$. Furthermore,
\begin{align} \label{cond:K:3}
\sup_{s\ge 0}\frac{|K'(s)|}{K(s)}<\infty.
\end{align}
\end{assumption}
\begin{remark}\label{rem:K}
We note that conditions~\eqref{cond:K:1}-\eqref{cond:K:2} are quite standard and can be found in literature \cite{bonaccorsi2012asymptotic,clement1997white,nunziato1971heat,
pruss2013evolutionary}. Condition~\eqref{cond:K:3} implies that 
\begin{align*}
-K'(t)\le c\,K(t),\quad
 t\ge 0,
\end{align*}
for some positive constant $c$. Taking integral on $[t,\infty)$ yields
\begin{align}\label{ineq:K/rho<c}
K(s)\le c\int_s^\infty\close K(r)\emph{d} r= c\rho(s).
\end{align}
As a consequence, $ \M^r\subseteq L^2_K(\rbb^+;H^{r+1})$, i.e., if $\eta\in\M^r$,~\eqref{ineq:K/rho<c} implies
\begin{align*}
\int_0^\infty \close K(s)\|\eta(s)\|^2_{H^{r+1}}\emph{d} s\le c\|\eta\|^2_{\M^r}.
\end{align*}
\end{remark}

Concerning the noise term, we assume that $w(t)$ is a cylindrical Wiener process of the form
\begin{align*}
w(t)=\sum_{k\ge 1} e_kB_k(t),
\end{align*}
where $\{e_k\}_{k\ge 1}$ are as in~\eqref{eqn:Ae_k=alpha_k.e_k} and $\{B_k(t)\}_{k\ge 1}$ is a sequence of mutually independent Brownian motions, all defined on the same stochastic basis $(\Omega,\Fcal,\{\Fcal_t\}_{t\ge 0},\P)$. With regard to the operator $Q$, we impose the following condition: \cite{cerrai2020convergence,glatt2017unique}
\begin{assumption}\label{cond:Q:well-posed}
The operator $Q:H\to H$ is diagonalized by $\{e_k\}_{k\ge 1}$, i.e., there exists a sequence $\{\lambda_k\}_{k\ge 1}$ such that
\begin{align} \label{cond:Q:1}
Qe_k=\lambda_k e_k.
\end{align}
Furthermore, 
\begin{equation} \label{cond:Q:2}
\emph{Tr}(AQQ^*)=\sum_{k\ge 1}\lambda_k^2\alpha_k<\infty.
\end{equation}
\end{assumption}

Concerning the nonlinearity, we will assume the following  standard conditions: \cite{caraballo2007existence,caraballo2008pullback,
cerrai2020convergence}

\begin{assumption} \label{cond:phi:well-possed}
The function $\f:\rbb\to\rbb$ is $C^1$ satisfying $\f(0)=0$. Moreover, the followings hold:

1. There exist $p>0,\, a_1>0$ such that
\begin{equation} \label{cond:phi:1}
 |\f(x)| \le a_1(1+|x|^p), \quad x\in\rbb.
\end{equation}

2. There exist $a_2,\,a_3>0$ such that
\begin{equation} \label{cond:phi:2}
x\f(x)\le -a_2|x|^{p+1}+a_3, \quad x\in\rbb,
\end{equation}
where $p$ is the same constant as in~\eqref{cond:phi:1}.

3. There exists $a_\f>0$ such that
\begin{equation}\label{cond:phi:3}
\sup_{x\in\rbb}\f'(x)=: a_\f<\infty.
\end{equation}
\end{assumption}

A concrete example of $\f$ is the class of odd--degree polynomials with negative leading coefficients, i.e., 
$$\f(x)=-c_{2n+1}x^{2n+1}+c_{2n}x^{2n}+\dots+c_0,$$
for some constants $n\ge 1$, $c_{2n+1}>0$ and $c_{2n},\dots,c_0\in\rbb$ are constants.

In light of relation~\eqref{eqn:eta:int-by-parts} together with Assumptions~\ref{cond:K}, \ref{cond:Q:well-posed} and~\ref{cond:phi:well-possed}, we are now in a position to define weak solutions for~\eqref{eqn:react-diff:K}.
\begin{definition}\label{def:well-posed}
For $\x\in\H$, a process $\Phi_\x(t)=(u_\x(t),\eta_\x(t))$ is called a weak solution for~\eqref{eqn:react-diff:K} if 
\begin{align*}
u_\x\in C(0,T;H)\cap L^2(0,T;H^1), \quad \f(u_\x)\in L^{q}(0,T;L^q),\quad \eta_\x\in C(0,T;\M),
\end{align*}
where $q$ is the H\"older conjugate of $p+1$ as in Assumption~\ref{cond:phi:well-possed}. Moreover, for all $t\in[0,T]$, $v\in H^1\cap L^{p+1}$ and $\etatilde\in \M$ such that $\partial_s\etatilde\in \M$
\begin{equation}\label{eqn:weak-solution:u}
\begin{aligned} 
\la u_\x(t),v\ra_{H}&=\la \pi_1\x,v\ra_H-\int_0^t\la A^{1/2}u_\x(r),v\ra_{H^1}\emph{d} r+\int_0^t \int_0^\infty\close K(s)\la A^{1/2} \eta_\x(r;s),v\ra_{H^1}\emph{d} s \emph{d} r\\
&\qquad\qquad+\int_0^t \la \f(u_\x(r)),v\ra_H\emph{d} r+\int_0^t \la v,Q\emph{d} w(r)\ra_H,
\end{aligned}
\end{equation}
\begin{equation}\label{eqn:weak-solution:eta:1}
\begin{aligned}
\la \eta_\x(t),\etatilde\ra_\M &=\la \pi_2\x,\etatilde\ra_\M+\int_0^t\rho(0)\la u_\x(r),\etatilde(0)\ra_{H^1}\emph{d} r + \int_0^t\la \eta_\x(r),\partial_s\etatilde\ra_{\M}\emph{d} r\\
&\qquad\qquad +\int_0^t\int_0^\infty\close\rho'(s)\la \eta_\x(r;s),\etatilde(s)\ra_{H^1}\emph{d} s\emph{d} r.
\end{aligned}
\end{equation}
\end{definition}

We now state the following proposition giving the existence of a solution for~\eqref{eqn:react-diff:K}.
\begin{proposition}\label{prop:well-posed}
Suppose that Assumptions~\ref{cond:K}, \ref{cond:Q:well-posed} and \ref{cond:phi:well-possed} hold. Then,~\eqref{eqn:react-diff:K} admits a unique solution in the sense of Definition~\ref{def:well-posed}. In particular,
\begin{align} \label{eqn:weak-solution:eta}
\eta(t;s)=u(t-s)\boldsymbol{1}\{0\le s\le t\}+\eta_0(s-t)\boldsymbol{1}\{s> t\}.
\end{align}
\end{proposition}

We remark that in~\cite{bonaccorsi2012asymptotic}, under a stronger assumption on the nonlinearities, the authors studied the notion of mild solutions for~\eqref{eqn:react-diff:K}. The existence of such solutions was established using a classical Yosida approximation for SPDEs that can be found in literature~\cite[Section 4.2]{bonaccorsi2012asymptotic} (see also \cite{cerrai2001second,da1996ergodicity,da2014stochastic}). Similar method was also employed in the work of~\cite{caraballo2007existence,caraballo2008pullback}. On the other hand, in this note, we will construct the weak solutions for~\eqref{eqn:react-diff:K} via a Galerkin approximation, following the framework in~\cite{glatt2008stochastic,robinson2001infinite}. The explicit construction will be presented later in Appendix~\ref{sec:well-posed}.

\subsection{Geometric ergodicity} We now turn to the main topic of the paper concerning statistically steady states of~\eqref{eqn:react-diff:K}. 

Given the well-posedness result in the previous subsection, we can thus introduce the Markov transition probabilities of the solution $\Phi_\x(t)$ by
\begin{align*}
P_t(\x,A):=\P(\Phi_\x(t)\in A),
\end{align*}
which are well-defined for $t\ge 0$, initial states $\x\in\H$ and Borel sets $A\subseteq \H$. Letting $\B_b(\H)$ denote the set of bounded Borel measurable functions $f:\H \rightarrow \rbb$, the associated Markov semigroup $P_t:\B_b(\H)\to\B_b(\H)$ is defined and denoted by
\begin{align}\label{form:P_t}
P_t f(\x)=\E[f(\Phi_\x(t))], \quad f\in \B_b(\H).
\end{align}

\begin{remark} \label{rem:Feller} We note that following the estimates in Appendix~\ref{sec:well-posed}, it can be shown that the solution $\Phi_\x(t)$ is continuous with respect to the initial condition $\x\in\H$, i.e.,
\begin{align*}
\Phi_{\x_n}(t)\to \Phi_\x(t) \text{  a.s. in }\H,  
\end{align*}
whenever $\x_n\to\x$ in $\H$. As a consequence, the Markov semigroup $P_t$ is Feller. That is $P_tf\in C_b(\H)$ for all $f\in C_b(\H)$ where $C_b(\H)$ represents the set of continuous bounded functions on $\H$ \cite{cerrai2001second,da1996ergodicity,da2014stochastic}.
\end{remark}

Recall that a probability measure $\mu\in Pr(\H)$ is said to be {\it\textbf{invariant}} for the semigroup $P_t$ if for every $f\in \B_b(\H)$
\begin{align*}
\int_{\H}P_t f(\x)\mu(\d\x)=\int_{\H} f(\x)\mu(\d\x).
\end{align*}
In literature, the existence of invariant probability measures is typically established via the Krylov-Bogoliubov argument combined with the tightness of a sequence of auxiliary probability measures \cite{bakhtin2005stationary,weinan2001gibbsian,hairer2011theory,
hairer2011asymptotic,
ito1964stationary}. However, as mentioned in Section~\ref{sec:functional-setting}, since the embedding of $\M^{r_1}\subset\M^{r_2}$, $r_1>r_2$, is only continuous \cite{conti2005singular,conti2006singular}, it is not clear whether under the same hypothesis of Proposition~\ref{prop:well-posed}, an invariant probability measure $\mu$ exists. Nevertheless, we are able to assert the following moment bounds of any such $\mu$.
\begin{theorem}\label{thm:mu:moment-bound}
Under the same hypothesis of Proposition~\ref{prop:well-posed}, any invariant probability measure $\mu$ for~\eqref{eqn:react-diff:K} must satisfy
\begin{equation} \label{ineq:mu:moment-bound}
\int_{\H}\exp\{\beta\|\x\|^2_{\H}\}+\|\pi_1\x\|^n_{H^1}\|\pi_1\x\|^2_{H^2}+\|\pi_2\x\|^n_{\M^1}\mu(\emph{d} \x)<\infty,
\end{equation}
for all $\beta>0$ sufficiently small and $n\ge 1$.
\end{theorem}
\noindent  The proof of Theorem~\ref{thm:mu:moment-bound} will be carried out in Section~\ref{sec:moment-bound:mu}. 

Following the framework of~\cite{butkovsky2020generalized,
cerrai2020convergence,hairer2006ergodicity,
hairer2008spectral,hairer2011asymptotic}, we recall that a function $d:\H\times\H\to[0,\infty)$ is called \emph{distance-like} if it is symmetric, lower semi-continuous, and $d(\x,\y)=0\Leftrightarrow \x=\y$ \cite[Definition 4.3]{hairer2011asymptotic}. Let $\W_d$ be the Wasserstein metric in $Pr(\H)$ associated with $d$, defined by
\begin{align} \label{form:W_d}
\W_d(\mu_1,\mu_2)=\sup_{[f]_{\text{Lip}}\leq1}\Big|\int_{\H}f(\x)\mu_1(\d \x)-\int_{\H}f(\x)\mu_2(\d\x)\Big|,
\end{align}
where
\begin{align*}
[f]_{\text{Lip}}=\sup_{\x\neq \y}\frac{|f(\x)-f(\y)|}{d(\x,\y)}.
\end{align*}
By the dual Kantorovich Theorem, it is well-known that
\begin{equation} \label{form:W_d:dual-Kantorovich}
\W_d(\mu_1,\mu_2) = \inf \E\, d(X,Y),
\end{equation}
where the infimum is taken over all pairs $(X,Y)$ such that $X\sim \mu_1$ and $Y\sim\mu_2$. In our settings, we will particularly pay attention to the following two distances in $\H$: the former is  the discrete metric, i.e., $d(\x,\y)=\mathbf{1}\{\x\neq \y\}$. The corresponding $\W_d$ is the usual total variation distance, denoted by $\W_{\TV}$. The latter is the distance $d_N$, $N>0$, given by \cite{butkovsky2020generalized,hairer2011asymptotic,kulik2017ergodic,kulik2015generalized}
\begin{equation} \label{form:d_N}
d_N(\x,\y):=N\|\x-\y\|_\H\mi 1,
\end{equation}
which we will employ to estimate the convergent rate of~\eqref{eqn:react-diff:K} toward equilibrium. The relation between $\W_{d_N}$ and $\W_{\TV}$ will be become clearer in Section~\ref{sec:ergodicity}, cf. Lemma~\ref{lem:W_(d_N)<W_TV}.

As mentioned in Section~\ref{sec:intro:overview}, given the generality of the potential $\f$, we will make use of the noise term that is sufficiently forced in many directions of the phase space. So that $\f$ can be dominated by the noise together with the Laplacian. More precisely, we make the following additional assumption on $Q$ and $\f$: \cite{cerrai2020convergence,glatt2017unique}
\begin{assumption}\label{cond:ergodicity}
Let $\f$ be as in Assumption~\ref{cond:phi:well-possed} and $\nbar\in\nbb$ be an index such that
\begin{equation} \label{cond:phi:ergodicity}
\sup_{x\in\rbb}\f'(x)=a_\f< \big[1-\|K\|_{L^1(\rbb^+)}\big]\alpha_{\nbar},
\end{equation}
where $\alpha_{\nbar}$ is the eigenvalue associated with $e_{\nbar}$ as in~\eqref{eqn:Ae_k=alpha_k.e_k}. There exists a positive constant  $a_Q$ such that
\begin{equation} \label{cond:Q:ergodicity}
\|Qu\|_H\ge a_Q\|P_{\nbar}u\|_H,\quad u\in H,
\end{equation}
where $Q$ is as in Assumption~\ref{cond:Q:well-posed}, and $P_{\nbar}$ is the projection onto $\{e_1,\dots,e_{\nbar}\}$ as in~\eqref{form:P_n.u}.
\end{assumption}

\begin{remark} \label{rem:ergodicity} We note that Assumption~\ref{cond:ergodicity} is actually a condition about the noise structure. Since $\sup_{x}\f'(x)=a_\f<\infty$, cf.~\eqref{cond:phi:3}, $\|K\|_{L^1(\rbb^+)}<1$, cf.~\eqref{cond:K:1}, and the sequence of eigenvalues $\{\alpha_n\}_{n\ge 1}$ as in~\eqref{form:A} is diverging to infinity, there always exists an index $\nbar$ such that condition~\eqref{cond:phi:ergodicity} holds. We then require that noise be forced in at least $e_k-$directions, $k=1,\dots,\nbar$, hence the condition~\eqref{cond:Q:ergodicity}.
\end{remark}

We are now in a position to state the main result of the paper ensuring geometric ergodicity of~\eqref{eqn:react-diff:K}.

\begin{theorem} \label{thm:ergodicity}
Under the same hypothesis of Proposition~\ref{prop:well-posed}, suppose Assumption~\ref{cond:ergodicity} holds. Then, \eqref{eqn:react-diff:K} admits a unique invariant measure $\mu$. Furthermore, there exists $N>0$ sufficiently large such that
\begin{equation} \label{ineq:ergodicity:1}
\W_{d_N}(P_t(\x,\cdot),\mu)\le C_1(1+\|\x\|^2_{\H})e^{-C_2t}, \quad t\ge 0,\, \x\in\H,
\end{equation}
where $d_N$ is as in~\eqref{form:d_N} and $C_1,C_2$ are positive constants independent of $\x$ and $t$.

\end{theorem}

\noindent The proof of Theorem~\ref{thm:ergodicity} will be presented in Section~\ref{sec:ergodicity}.

\section{A priori bounds of ~\eqref{eqn:react-diff:K}} \label{sec:moment-bound:solution}
Throughout the rest of the paper, $c$ and $C$ denote generic positive constants that may change from line to line. The main parameters that they depend on will appear between parenthesis, e.g., $c(T,q)$ is a function of $T$ and $q$.

In this section, we collect several useful a priori moment bounds on the solutions of~\eqref{eqn:react-diff:K}. These results will be employed to prove Theorems~\ref{thm:mu:moment-bound} and~\ref{thm:ergodicity} in later sections. 

We start off by setting 
\begin{equation}\label{form:eps_1.and.K_1}
\varepsilon_1:=\frac{1-\|K\|_{L^1(\rbb^+)}}{\|K\|_{L^1(\rbb^+)}},\quad \text{and}\quad K_1:=1-\big(1+\tfrac{1}{2}\varepsilon_1\big)\|K\|_{L^1(\rbb^+)}=\frac{1-\|K\|_{L^1(\rbb^+)} }{2}.
\end{equation}
Recalling condition~\eqref{cond:K:1}, we observe that $\varepsilon_1$ and $K_1$ are both positive. In Lemma~\ref{lem:moment-bound:H} below, we assert two moment bounds in $\H$.
\begin{lemma} \label{lem:moment-bound:H}
Under the same hypothesis as in Proposition~\ref{prop:well-posed}, the followings hold:

1. \begin{equation}\label{ineq:moment-bound:H}
\E\|(u_\x(t),\eta_\x(t))\|^2_\H \le e^{-c t}\|\x\|^2_\H+C,
\end{equation}
for some positive constants $c$ and $C$ independent of initial condition $\x$ and time $t$.

2. For all
\begin{equation} \label{cond:beta}
%\beta\in\Big(0,\tfrac{K_1\alpha_1}{\|Q\|^2_{L(H)}}\Big),
\beta\in\big(0,2K_1\alpha_1/\|Q\|^2_{L(H)}\big),
\end{equation}
there exist positive constants $c=c(\beta)$ and $C=C(\beta)$ independent of $\x$ and $t$ such that
\begin{equation} \label{ineq:exponential-bound:H}
\E\, e^{\beta\|(u_\x(t),\eta_\x(t))\|_\H^2} \le e^{-ct} e^{\beta \|\x\|^2_\H}+C.
\end{equation}
\end{lemma}
\begin{proof} To simplify notations, throughout the proof, we will omit the subscript $\x$ in $u_\x$ and $\eta_\x$. 

Denote by $g$ the function given by
\begin{equation} \label{form:g}
g(u,\eta)=\tfrac{1}{2}\|(u,\eta)\|_\H^2=\tfrac{1}{2}\|u\|^2_H+\tfrac{1}{2}\|\eta\|^2_\M.
\end{equation}
A routine calculation gives
\begin{align*}
\d g(u(t),\eta(t))& =  -\|u(t)\|^2_{H^1}\d t+ \int_0^\infty\close  K(s)\la \eta(t;s),u(t)\ra_{H^1}\d s\d t+\la -\partial_s\eta,\eta\ra_\M\d t\\
&\qquad+\la \f(u(t),u(t)\ra_H\d t+\la u(t),Q\d w(t)\ra_H+\tfrac{1}{2} \Tr(QQ^*) \d t.
\end{align*}
Recalling condition~\eqref{cond:Q:2}, we have
\begin{align*}
\Tr(QQ^*)=\sum_{k\ge 1}\lambda_k^2<\infty.
\end{align*}
Concerning the convolution involving $K$, by Young's inquality,
\begin{align*}
\int_0^\infty\close  K(s)\la \eta(t;s), u(t)\ra_{H^1}\d s&\le \tfrac{1+\varepsilon_1}{2}\|K\|_{L^1(\rbb^+)}\| u(t)\|^2_{H^1}+\tfrac{1}{2(1+\varepsilon_1)}\int_0^\infty\close  K(s)\|\eta(t;s)\|^2_{H^1}\d s\\
&= \tfrac{1+\varepsilon_1}{2}\|K\|_{L^1(\rbb^+)}\| u(t)\|^2_{H^1}-\tfrac{1}{2(1+\varepsilon_1)}\int_0^\infty\close  \rho'(s)\|\eta(t;s)\|^2_{H^1}\d s,
\end{align*}
where $\varepsilon_1$ is as in~\eqref{form:eps_1.and.K_1}. Recalling the identity~\eqref{eqn:<partial_s.eta,eta>} and the fact that $\eta(t;0)=u(t)$, it holds that
\begin{align*}
\la -\partial_s\eta(t),\eta(t)\ra_\M &= \tfrac{1}{2}\|K\|_{L^1(\rbb)}\|u(t)\|^2_{H^1}+\tfrac{1}{2}\int_0^\infty\close  \rho'(s)\|\eta(t;s)\|^2_{H^1}\d s.
\end{align*}
Also, since $K$ satisfies~\eqref{cond:K:2} and $\rho(t)=\int_t^\infty K(s)\d s$, it holds that
\begin{align}\label{cond:rho}
\rho'(t)\le -\delta \rho(t),\quad t\ge 0.
\end{align}
Recalling $K_1$ as in~\eqref{form:eps_1.and.K_1}, we then derive the bound
\begin{align}
 &-\|u(t)\|^2_{H^1}+ \int_0^\infty\close  K(s)\la \eta(t;s),u(t)\ra_{H^1}\d s+\la -\partial_s\eta,\eta\ra_\M \nonumber\\
 &\le -\Big[1-\big(1+\tfrac{\varepsilon_1}{2}\big)\|K\|_{L^1(\rbb^+)}\Big]\| u(t)\|^2_{H^1}+\tfrac{\varepsilon_1}{2(1+\varepsilon_1)}\int_0^\infty\close  \rho'(s)\|\eta(t;s)\|^2_{H^1}\d s\nonumber\\
 &\le -\Big[1-\big(1+\tfrac{\varepsilon_1}{2}\big)\|K\|_{L^1(\rbb^+)}\Big]\| u(t)\|^2_{H^1}-\tfrac{\varepsilon_1\delta}{2(1+\varepsilon_1)}\int_0^\infty\close  \rho(s)\|\eta(t;s)\|^2_{H^1}\d s\nonumber\\
 &= -K_1\| u(t)\|^2_{H^1}-\tfrac{\varepsilon_1\delta}{2(1+\varepsilon_1)}\|\eta(t)\|^2_\M \label{ineq:u+int.K<eta,u>:H1} \\
 &\le -K_1\alpha_1\| u(t)\|^2_{H}-\tfrac{\varepsilon_1\delta}{2(1+\varepsilon_1)}\|\eta(t)\|^2_\M. \label{ineq:u+int.K<eta,u>:H}
\end{align}
To bound the non-linear term, we invoke~\eqref{cond:phi:3} to see that
\begin{align*}
\la \f(u(t)),u(t)\ra_H\le -a_2\|u(t)\|^{p+1}_{L^{p+1}}+a_3|\domain|\le a_3|\domain|.
\end{align*}
In the above, $|\domain|$ denotes the volume of $\domain$ in $\rbb^d$. Collecting everything now yields the estimate
\begin{align} \label{ineq:d.E[g(t)]}
\tfrac{\d}{\d t}\E g(u(t),\eta(t))\le -K_1\alpha_1\E\| u(t)\|^2_{H}-\tfrac{\varepsilon_1\delta}{2(1+\varepsilon_1)}\E\|\eta(t)\|^2_\M+a_3|\domain|+\tfrac{1}{2}\sum_{k\ge 1}\lambda_k^2,
\end{align}
which proves~\eqref{ineq:moment-bound:H} by virtue of Gronwall's inequality.

With regard to~\eqref{ineq:exponential-bound:H}, for $\beta>0$ to be chosen later, the partial derivatives of $e^{\beta g}$ along the direction of $\xi\in \H$ is given by
\begin{align*}
\la D_u e^{\beta g},\xi\ra_\H =\kappa e^{\beta g}\la u,\pi_1\xi\ra_H,\quad\la D_\eta e^{\beta g},\xi\ra_\H &=\kappa e^{\beta g}\la \eta,\pi_2\xi\ra_\M,\end{align*}
and
\begin{align*}
 D_{uu} e^{\beta g}(\xi) &= \beta e^{\beta g} \pi_1\xi+\beta^2 e^{\beta g} \la u,\pi_1\xi\ra_{H}\pi_1\xi.
\end{align*}
Then, by Ito's formula,
\begin{align*}
\d\, e^{\beta g(u(t),\eta(t))} & = \beta e^{\beta g(u(t),\eta(t))}\Big( -\|u(t)\|^2_{H^1}\d t+ \int_0^\infty\close  K(s)\la \eta(t;s),u(t)\ra_{H^1}\d s\d t+\la -\partial_s\eta,\eta\ra_\M\d t\\
&\qquad+\la \f(u(t),u(t)\ra_H\d t+\la u(t),Q\d w(t)\ra_H+\tfrac{1}{2} \sum_{k\ge 1}\lambda_k^2 \d t+\tfrac{1}{2}\beta\sum_{k\ge 1} \lambda_k^2|\la u(t),e_k\ra_H|^2\d t\Big),
\end{align*}
where $\lambda_k$ is as in Assumption~\ref{cond:Q:well-posed}. We note that
\begin{align*}
\sum_{k\ge 1} \lambda_k^2|\la u(t),e_k\ra_H|^2\le \|Q\|^2_{L(H)}\|u(t)\|^2_H.
\end{align*}
Together with the estimates as in the proof of~\eqref{ineq:moment-bound:H}, we arrive at the bound
\begin{align*}
&\tfrac{\d}{\d t}\E \,e^{\beta g(u(t),\eta(t))}\\ &\le \E\,\beta  e^{\beta g(u(t),\eta(t))} \Big(-\big[K_1\alpha_1-\tfrac{1}{2}\beta\|Q\|^2_{L(H)}\big] \|u(t)\|^2_{H}-\tfrac{\varepsilon_1\delta}{2(1+\varepsilon_1)}\|\eta(t)\|^2_\M+a_3|\domain|+\tfrac{1}{2} \sum_{k\ge 1}\lambda_k^2\Big).
\end{align*}
Since $\beta\in(0,2K_1\alpha_1/\|Q\|^2_{L(H)})$, we infer 
\begin{align*}
\tfrac{\d}{\d t}\E \,e^{\beta g(u(t),\eta(t))} &\le \E e^{\beta g(u(t),\eta(t))}(-c\,g(u(t),\eta(t))+C),
\end{align*}
for some positive constants $c$ and $C$ independent of $t$. We now employ the elementary inequality
\begin{align*}
e^{\beta r}(-cr+C)\le -\widetilde{c} e^{\beta r}+\widetilde{C},\qquad r\ge 0,
\end{align*}
to arrive at the bound
\begin{align*}
\tfrac{\d}{\d t}\E \,e^{\beta g(u(t),\eta(t))} &\le -c\,\E e^{\beta g(u(t),\eta(t))}+C,
\end{align*}
whence 
\begin{align*}
\E \,e^{\beta g(u(t),\eta(t))}  \le e^{-c t}\E \,e^{\beta g(\x)}+C.
\end{align*}
In the above, $c$ and $C$ are positive constants independent of initial condition $\x$ and time $t$. This establishes~\eqref{ineq:exponential-bound:H}, thereby concluding the proof.
\end{proof}

Next, we state and prove Lemma~\ref{lem:moment-bound:H1} giving moment bounds in $\H^1$ of the solutions.

\begin{lemma}\label{lem:moment-bound:H1}
For all $\x\in\H^1$, $n\ge 2$, and $t\ge 0$, it holds that
\begin{equation}\label{ineq:moment-bound:H1}
\E \|(u_\x(t),\eta_\x(t))\|^{2n}_{\H^1}+\int_0^t \|u_\x(r)\|^{2n-2}_{H^1}\|Au_\x(r)\|^2_H+\|\eta_\x(r)\|^{2n}_{\M^1}\emph{d} r \le C(\|\x\|^{n}_{\H^1}+t), 
\end{equation}
and
\begin{equation}\label{ineq:expo-decay:H1:expo-decay}
\E \|(u_\x(t),\eta_\x(t))\|^{2n}_{\H^1} \le  e^{-ct}\|\x\|^{2n}_{\H^1}+C\,e^{\beta\|\x\|^2_\H}+C,
\end{equation}
where $\beta$ is as in~\eqref{cond:beta}, $c=c(n,\beta)$ and $C=C(n,\beta)$ are positive constants independent of $\x$ and $t$.
\end{lemma}
\begin{proof}
We proceed to prove~\eqref{ineq:moment-bound:H1} by induction on $n$. 

We first start with the base case $n=2$ and set
\begin{equation}\label{form:g_1}
g_1(u,\eta)=\tfrac{1}{2}\big(\|u\|^2_{H^1}+\|\eta\|^2_{\M^1}\big),
\end{equation}
A routine calculation yields
\begin{equation} \label{eqn:Ito:g_1}
\begin{aligned}
\d g_1(u(t),\eta(t))& = -\|Au(t)\|^2_{H}\d t+ \int_0^\infty\close  K(s)\la A\eta(t;s),Au(t)\ra_{H}\d s\d t+\la -\partial_s\eta,\eta\ra_{\M^1}\d t\\
&\qquad+\la \f'(u(t)\grad u(t),\grad u(t)\ra_H\d t+\la u(t),Q\d w(t)\ra_{H^1}+\tfrac{1}{2} \Tr(AQQ^*)\d t.
\end{aligned}
\end{equation}
In the above,
\begin{align*}
\Tr(AQQ^*)=\sum_{k\ge 1}\lambda_k^2\alpha_k<\infty,
\end{align*}
by virtue of condition~\eqref{cond:Q:2}. Similarly to~\eqref{eqn:<partial_s.eta,eta>}, we have
\begin{align*}
\la -\partial_s\eta(t),\eta(t)\ra_{\M^1} &=-\tfrac{1}{2}\int_0^\infty\close  \rho(s)\partial_s\|A\eta(t;s)\|^2_H\d s\\
&= \tfrac{1}{2}\|K\|_{L^1(\rbb)}\|Au(t)\|^2_{H}+\tfrac{1}{2}\int_0^\infty\close  \rho'(s)\|A\eta(t;s)\|^2_{H}\d s,
\end{align*}
where in the last implication above, we employed again the fact that $\eta(t;0)=u(t)$. Also,
\begin{align*}
\int_0^\infty\close  K(s)\la A\eta(t;s), Au(t)\ra_{H}\d s&\le \tfrac{1+\varepsilon_1}{2}\|K\|_{L^1(\rbb^+)}\| Au(t)\|^2_{H}+\tfrac{1}{2(1+\varepsilon_1)}\int_0^\infty\close  K(s)\|A\eta(t;s)\|^2_{H}\d s\\
&= \tfrac{1+\varepsilon_1}{2}\|K\|_{L^1(\rbb^+)}\|A u(t)\|^2_{H}-\tfrac{1}{2(1+\varepsilon_1)}\int_0^\infty\close  \rho'(s)\|A\eta(t;s)\|^2_{H}\d s,
\end{align*}
where $\varepsilon_1$ is given by~\eqref{form:eps_1.and.K_1}. Using an argument similarly to~\eqref{ineq:u+int.K<eta,u>:H}, we have the bound
\begin{align*}
 &-\|Au(t)\|^2_{H}+ \int_0^\infty\close  K(s)\la A\eta(t;s),Au(t)\ra_{H}\d s+\la -\partial_s\eta,\eta\ra_{\M^1}\\
 &\le -K_1\| Au(t)\|^2_{H}-\tfrac{\varepsilon_1\delta}{2(1+\varepsilon_1)}\|\eta(t)\|^2_{\M^1},
\end{align*}
where $\delta$, $K_1$ are respectively as in~\eqref{cond:K:2},~\eqref{form:eps_1.and.K_1} and $\alpha_1$ is the first eigenvalue of $A$. With regard to the nonlinear term in~\eqref{eqn:Ito:g_1}, we invoke Assumption~\ref{cond:phi:well-possed} to see that
\begin{align*}
\la \f'(u(t))\grad u(t),\grad u(t)\ra_H \le a_\f\la \grad u(t),\grad u(t)\ra_H =a_\f\la u(t),Au(t)\ra_H\le \frac{a_\f^2}{2K_1}\|u(t)\|^2_H+\tfrac{1}{2}K_1\|Au(t)\|^2_H.
\end{align*}
Combining the above estimates with~\eqref{eqn:Ito:g_1}, we arrive at the moment bound
\begin{align} \label{ineq:d.E(u,eta):H1}
\tfrac{\d }{\d t}\E \|(u(t),\eta(t))\|^2_{\H^1}&+K_1\E\|Au(t)\|^2_{H}+\tfrac{\varepsilon_1\delta}{1+\varepsilon_1}\E\|\eta(t)\|^2_{\M^1}\le \tfrac{a_\f^2}{K_1} \E\|u(t)\|^2_{H}+\sum_{k\ge 1}\lambda_k^2\alpha_k.
\end{align}
To estimate the term involving $\|u(t)\|_H^2$ on the above right-hand side, we invoke~\eqref{ineq:moment-bound:H} to see that
\begin{align*}
 \E\|u(t)\|^2_{H} \le  e^{-ct}\|\x\|^2_\H+C.
\end{align*}
We then infer the estimate
\begin{align*}
\E \|(u(t),\eta(t))\|^2_{\H^1}&+\int_0^t K_1\E\|Au(r)\|^2_{H}+\tfrac{\varepsilon_1\delta}{1+\varepsilon_1}\E\|\eta(r)\|^2_{\M^1}\le C\big(\|\x\|^2_{\H^1}+e^{\beta\|\x\|^2_\H}+t\big),
\end{align*}
and
\begin{align*}
\E \|(u(t),\eta(t))\|^2_{\H^1}&\le e^{-ct}\|\x\|^2_{\H^1}+C\int_0^t e^{-c(t-r)} \Big(\E\|u(t)\|^2_{H} +1\Big)\d r\\
&\le e^{-ct}\|\x\|^2_{\H^1}+C\,e^{\beta\|\x\|^2_\H}+C,
\end{align*}
thereby finishing the proof of~\eqref{ineq:moment-bound:H1}-\eqref{ineq:expo-decay:H1:expo-decay} for $n=1$. 

Now assume~\eqref{ineq:moment-bound:H1}-\eqref{ineq:expo-decay:H1:expo-decay} hold for up to $(n-1)\ge 1$. Let us consider the case $n$. We first compute partial derivatives of $g_1(u,\eta)^n$ in $\H^1$:
\begin{align*}
\la D_u g_1^n,\xi\ra_{\H^1} =n g_1^{n-1} \la u,\pi_1\xi\ra_{H^1},\quad\la D_\eta g_1^n,\xi\ra_{\H^1} &=ng_1^{n-1}\la \eta,\pi_2\xi\ra_{\M^1},
\end{align*}
and
\begin{align*}
 D_{uu} (g_1^{n})(\xi) &= ng_1^{n-1} \pi_1\xi+ n(n-1) g_1^{n-2}\la u,\pi_1\xi\ra_{H^1}u.
\end{align*}
By Ito's formula, the following holds
\begin{align*}
&\d g_1(u(t),\eta(t))^n \\
&=n g_1(u(t),\eta(t))^{n-1}\Big(-\|Au(t)\|^2_{H}\d t+ \int_0^\infty\close  K(s)\la A\eta(t;s),Au(t)\ra_{H}\d s\d t+\la -\partial_s\eta,\eta\ra_{\M^1}\d t\\
&\qquad+\la \f'(u(t)\grad u(t),\grad u(t)\ra_H\d t+\la u(t),Q\d w(t)\ra_{H^1}+\tfrac{1}{2} \Tr(AQQ^*)\d t\Big)\\
&\qquad +\tfrac{1}{2}n(n-1)g_1(u(t),\eta(t))^{n-2}\sum_{k\ge 1}\lambda_k^2|\la u(t),e_k\ra_{H^1}|^2 .
\end{align*}
Similarly to the estimates in the base case $n=1$, we readily have
\begin{align*}
&-\|Au(t)\|^2_{H}+ \int_0^\infty\close  K(s)\la A\eta(t;s),Au(t)\ra_{H}\d s+\la -\partial_s\eta,\eta\ra_{\M^1}+\la \f'(u(t)\grad u(t),\grad u(t)\ra_H\\
&\le -\tfrac{1}{2}K_1\| Au(t)\|^2_{H}-\tfrac{\varepsilon_1\delta}{2(1+\varepsilon_1)}\|\eta(t)\|^2_{\M^1}+\frac{a_\f^2}{2K_1}\|u(t)\|^2_H\\
&\le -\tfrac{1}{4}K_1\| Au(t)\|^2_{H}-\tfrac{1}{4}K_1\alpha_1\| u(t)\|^2_{H^1}-\tfrac{\varepsilon_1\delta}{2(1+\varepsilon_1)}\|\eta(t)\|^2_{\M^1}+\tfrac{a_\f^2}{2K_1}\|u(t)\|^2_H
\end{align*}
Using $\varepsilon$-Young's inequality, we note that
\begin{align*}
g_1(u(t),\eta(t))^{n-1}\|u(t)\|^2_H\le \tfrac{n}{n-1}\varepsilon^{\frac{n}{n-1}} g_1(u(t),\eta(t))^{n}+n\varepsilon^{-n}\|u(t)\|^{2n}_H.
\end{align*}
Likewise, 
\begin{align*}
\tfrac{1}{2}g_1(u(t),\eta(t))^{n-2}\sum_{k\ge 1}\lambda_k^2|\la u(t),e_k\ra_{H^1}|^2&\le \|Q\|^2_{L(H)}g_1(u(t),\eta(t))^{n-2}\cdot \tfrac{1}{2}\|u(t)\|^2_{H^1}\\
&\le \|Q\|^2_{L(H)}g_1(u(t),\eta(t))^{n-1}\\
&\le \tfrac{n}{n-1}\varepsilon^{\frac{n}{n-1}} g_1(u(t),\eta(t))^{n}+n\varepsilon^{-n}\|Q\|^{2n}_{L(H)}.
\end{align*}
By taking $\varepsilon$ small enough, we arrive at the following useful estimate in expectation
\begin{equation}\label{ineq:E.g_1}
\begin{aligned} 
\tfrac{\d}{\d t} \E g_1(u(t),\eta(t))^n &\le -c\, \E g_1(u(t),\eta(t))^n -c\,\E g_1(u(t),\eta(t))^{n-1}\big(\|Au(t)\|^2_H+\|\eta(t)\|^2_{\M^1}\big)\\
&\qquad\qquad+C\E\|u(t)\|^{2n}_H+C.
\end{aligned}
\end{equation}
In view of Lemma~\ref{lem:moment-bound:H}, 
\begin{align} \label{ineq:E.|u|^n_H}
\E\|u(t)\|^{2n}_H\le Ce^{-ct}e^{\beta \|\x\|^2_\H}+C.
\end{align}
Combining~\eqref{ineq:E.|u|^n_H} and~\eqref{ineq:E.g_1} and integrating with respect to time, we obtain
\begin{align*}
\E g_1(u(t),\eta(t))^n + c\int_0^t\E g_1(u(r),\eta(r))^{n-1}\big(\|Au(r)\|^2_H+\|\eta(r)\|^2_{\M^1}\big)\d r &\le g_1(\x)^n +C e^{\beta \|\x\|^2_\H}+Ct.
\end{align*}
Recalling $g_1$ as in~\eqref{form:g_1}, estimate~\eqref{ineq:moment-bound:H1} immediately follows from the above inequality. Also, by variation formula, 
\begin{align*}
\E g_1(u(t),\eta(t))^n \le e^{-ct}g_1(\x)^n +C e^{\beta \|\x\|^2_\H}+C,
\end{align*}
which proves~\eqref{ineq:expo-decay:H1:expo-decay}. The proof is thus finished.
\end{proof}

\section{Moment bounds of invariant probability measures} \label{sec:moment-bound:mu}

In this section, we provide the proof of Theorem~\ref{thm:mu:moment-bound} giving regularity of any invariant measure~$\mu$. Following the framework in \cite{glatt2021long}, we introduce the auxiliary system
\begin{equation} \label{eqn:react-diff:K:u.hat}
\begin{aligned}
\d \uhat(t)&=-A \uhat(t)\d t +\int_{0}^{\infty}\close K(s) A\etahat(t;s)\d s\d t+\f(\uhat(t))\d t+Q\d w(t)\\
&\qquad-K_1\alpha_{n_*}P_{n_*}(\uhat(t)-u(t))\d t,\\
\d \etahat(t)&= -\partial_s\etahat(t)\d t, \quad(\uhat(0),\etahat(0))=0\in \H,\, \etahat(t;0)=\uhat(t),\, t>0.
\end{aligned}
\end{equation}
In the above, $K_1$ is as in~\eqref{form:eps_1.and.K_1}, $\alpha_{n_*}$ is the eigenvalue of $A$ associated with $e_{n_*}$ as in~\eqref{eqn:Ae_k=alpha_k.e_k}, and $P_{n_*}$ is the projection on to $\{e_1,\dots,e_{n_*}\}$ as in~\eqref{form:P_n.u}. Also, we chose $n_*$ sufficiently large such that 
\begin{align} \label{cond:K_1.alpha_n*>a_phi}
\sup_{x\in\rbb}\f'(x)=a_\f<K_1\alpha_{n_*}=\tfrac{1}{2}\big[1-\|K\|_{L^1(\rbb^+)}\big]\alpha_{n_*}.
\end{align}
\begin{remark} \label{rem:noise:regularity}
We note that the choice of $n_*$ as in~\eqref{cond:K_1.alpha_n*>a_phi} is always valid owing to the fact that $\sup_{x}\f'(x)<\infty$, cf.~\eqref{cond:phi:3}, $\|K\|_{L^1(\rbb^+)}<1$, cf.~\eqref{cond:K:1}, and $\{\alpha_n\}_{n\ge 1}$ is diverging to infinity. Although~\eqref{cond:K_1.alpha_n*>a_phi} seems similarly to~\eqref{cond:phi:ergodicity} as in  Assumption~\ref{cond:ergodicity} for geometric ergodicity, cf. Remark~\ref{rem:ergodicity}, for higher regularity properties of invariant probability measures, we do not require noise be forced in enough many directions. Only the same conditions for the well-posedness as in Proposition~\ref{prop:well-posed} are sufficient to prove Theorem~\ref{thm:mu:moment-bound}. 
\end{remark}

It is worth to point out that system~\eqref{eqn:react-diff:K:u.hat} only differs from~\eqref{eqn:react-diff:K} by the appearance of the linear term $-K_1\alpha_{n_*}P_{n_*}(\uhat(t)-u(t))\d t$. More importantly, since~\eqref{eqn:react-diff:K:u.hat} starts from the origin, it enjoys better regularity compared with~\eqref{eqn:react-diff:K}. This is precisely summarized in the following lemma whose proof will be deferred to the end of this section.

\begin{lemma}\label{lem:u.hat}
Let $(\uhat(t),\etahat(t))$ be the process as in~\eqref{eqn:react-diff:K:u.hat}. For all $\x\in \H$ and $t\ge 0$, the followings hold:
\begin{equation}\label{ineq:u-uhat}
\|(u(t)-\uhat(t),\eta(t)-\etahat(t))\|^2_{\H}\le e^{-ct}\|\x\|^2_\H,
\end{equation}
and
\begin{equation} \label{ineq:uhat}
\E\|(\uhat(t),\etahat(t))\|^2_{\H^1}\le C\|\x\|^2_\H+C,
\end{equation}
for some positive constants $c$ and $C$ independent of $\x$ and $t$.
\end{lemma}

Assuming the above result, we are now ready to conclude Theorem~\ref{thm:mu:moment-bound}.

\begin{proof}[Proof of Theorem~\ref{thm:mu:moment-bound}]
We first show that
\begin{align} \label{ineq:mu:exponential-bound:H}
\int_\H e^{\beta\|\x\|^2_\H}\mu(\d\x)<\infty,
\end{align}
for all $\beta\in(0,2K_1\alpha_1/\|Q\|^2_{L(H)})$, cf.~\eqref{cond:beta}. To see this, for $\varepsilon>0$, consider $R=R(\varepsilon)>0$ such that
\begin{align*}
\mu(B_R^c)<\varepsilon,
\end{align*}
where
\begin{align*}
B_R=\{x\in\H:\|\x\|_\H\le R\}.
\end{align*}
Given $N>0$, we set $\phi_N(\x)=e^{\beta\|\x\|^2_\H}\mi N$. By invariance of $\mu$, since $\phi_N$ is bounded,
\begin{align*}
\int_\H P_t\phi_N(\x)\mu(\d\x)=\int_\H \phi_N(\x)\mu(\d\x).
\end{align*}
Also, by the choice of $B_R$, we have
\begin{align*}
\int_\H P_t\phi_N(\x)\mu(\d\x)& =\int_{B_R} P_t\phi_N(\x)\mu(\d\x)+\int_{B^c_R} P_t\phi_N(\x)\mu(\d\x)\le \int_{B_R} P_t\phi_N(\x)\mu(\d\x)+N\varepsilon.
\end{align*}
Considering $P_t\phi_N(\x)$, we invoke~\eqref{ineq:exponential-bound:H} to see that
\begin{align*}
P_t\phi_N(\x)=\E\big(e^{\beta\|(u(t),\eta(t))\|^2_\H}\mi N\big) \le e^{-ct}e^{\beta\|\x\|^2_\H}+C,
\end{align*}
whence
\begin{align*}
\int_{B_R} P_t\phi_N(\x)\mu(\d\x)\le e^{-ct}e^{\beta R^2}+C.
\end{align*}
It follows that for all $N$ we have the bound
\begin{align*}
\int_\H \phi_N(\x)\mu(\d\x)\le e^{-ct}e^{\beta R^2}+N\varepsilon+C.
\end{align*}
We may take $\varepsilon$ small and then take t sufficiently large to arrive at the following uniform bound in $N$:
\begin{align*}
\int_\H \phi_N(\x)\mu(\d\x)\le C<\infty.
\end{align*}
The Dominated Convergence Theorem then implies~\eqref{ineq:mu:exponential-bound:H}.

Next, we aim to show that $\mu$ must concentrate in $\H^1$. To do so, for $R,\,N$, we set $$\psi_{R,N}(u,\eta)=R\mi (\|P_Nu\|_{H^1}+\|P_N\eta\|_{\M^1}).$$
We invoke invariance of $\mu$ again to see that
\begin{align} \label{eqn:psi_RN}
\int_\H P_t\psi_{R,N}(\x)\mu(\d\x)=\int_\H \phi_{R,N}(\x)\mu(\d\x).
\end{align}
Recalling the pair $(\uhat(t),\etahat(t))$ solving the auxiliary system~\eqref{eqn:react-diff:K:u.hat}, we estimate as follows:
\begin{align*}
P_t\psi_{R,N}(\x)\le \E \|P_N(u(t)-\uhat(t))\|_{H^1}+\E\|P_N(\eta(t)-\etahat(t))\|_{\M^1}+ \E \|P_N\uhat(t)\|_{H^1}+\E\|P_N\etahat(t)\|_{\M^1}.
\end{align*}
In light of ~\eqref{ineq:u-uhat} and Sobolev embedding, we have
\begin{align*}
&\E \|P_N(u(t)-\uhat(t))\|_{H^1}+\E\|P_N(\eta(t)-\etahat(t))\|_{\M^1}\\
&\le \alpha_N^{1/2}\big(\E \|u(t)-\uhat(t)\|_{H}+\E\|\eta(t)-\etahat(t)\|_{\M}\big)\le \alpha_N^{1/2} e^{-ct}\|\x\|_{\H}.
\end{align*}
Also, we invoke~\eqref{ineq:uhat} to see that
\begin{align*}
 \E \|P_N\uhat(t)\|_{H^1}+\E\|P_N\etahat(t)\|_{\M^1} &\le  \E \|\uhat(t)\|_{H^1}+\E\|\etahat(t)\|_{\M^1}\le C\|\x\|_{\H}+C.
\end{align*}
It follows that
\begin{align*}
P_t\psi_{R,N}(\x)\le \alpha_N^{1/2} e^{-ct}\|\x\|_{\H}+C\|\x\|_{\H}+C,
\end{align*}
where $c$ and $C$ are independent of $R,N,\x$ and $t$. Combining with~\eqref{eqn:psi_RN}, we infer that
\begin{align*}
\int_\H \phi_{R,N}(\x)\mu(\d\x) \le \alpha_N^{1/2} e^{-ct}\int_\H\|\x\|_{\H}\mu(\d\x)+C\int_\H\|\x\|_{\H}\mu(\d\x)+C.
\end{align*}
By virtue of~\eqref{ineq:mu:exponential-bound:H}, we readily have $\int_\H\|\x\|_{\H}\mu(\d\x)<\infty$, implying
\begin{align*}
\int_\H \phi_{R,N}(\x)\mu(\d\x)\le \alpha_N^{1/2} e^{-ct}C+C.
\end{align*}
Since the above inequality holds for arbitrarily $t>0$, we may take $t$ sufficiently large so that $\alpha_N^{1/2} e^{-ct}<1$. As a consequence, we obtain the following uniform bound in $R,N$
\begin{align*}
\int_\H \phi_{R,N}(\x)\mu(\d\x)\le C.
\end{align*}
We invoke the Monotone Convergence Theorem again to obtain 
\begin{align*}
\int_\H \|u\|_{H^1}+\|\eta\|_{\M^1}\mu(\d u,\d\eta)<\infty,
\end{align*}
which proves that $\mu(\H^1)=1$.

We now turn to higher moment bounds in $\H^1$ for $\mu$. For any $n\ge 1$ and $\x\in\H^1$, recall~\eqref{ineq:expo-decay:H1:expo-decay} that
\begin{align*}
\E\|(u(t),\eta(t))\|^{2n}_{\H^1}\le e^{-ct}\|\x\|^{2n}_{\H^1}+Ce^{\beta\|\x\|^2_\H}+C.
\end{align*}
Using an argument similarly to the proof of~\eqref{ineq:mu:exponential-bound:H}, we deduce the bound
\begin{align}  \label{ineq:mu:moment-bound:H1}
\int_{\H}\|\x\|^{2n}_{\H^1}\mu(\d \x)<\infty.
\end{align}
With regard to $\int_\H \|u\|^{2{n-1}}_{H^1}\|Au\|^2_H\mu(\d u,\d\eta)$, we employ~\eqref{ineq:moment-bound:H1} to see that
\begin{align*}
\int_\H \big(\|u\|^{2{n-1}}_{H^1}\|Au\|^2_H\mi R \big)\mu(\d u,\d\eta)&=\frac{1}{t}\int_\H\int_0^t \E\Big(\|u(r)\|^{2n-2}_{H^1}\|Au(r)\|^2_H\mi R\Big)\mu(\d \x)\\
&\le   \frac{C}{t}\Big(\int_\H\|\x\|^{2n}_{\H^1}\mu(\d\x)+t\Big)\\
&=C\int_\H\|\x\|^{2n}_{\H^1}\mu(\d\x)+C<\infty,
\end{align*}
whence
\begin{align} \label{ineq:mu:moment-bound:H1:u}
\int_\H \|u\|^{2{n-1}}_{H^1}\|Au\|^2_H\mu(\d u,\d\eta)<\infty,
\end{align}
by virtue of the Monotone Convergence Theorem. 

Finally, we combine~\eqref{ineq:mu:exponential-bound:H},~\eqref{ineq:mu:exponential-bound:H} and~\eqref{ineq:mu:moment-bound:H1:u} to deduce~\eqref{ineq:mu:moment-bound}. The proof is thus finished.
\end{proof}

We now give the proof of Lemma~\ref{lem:u.hat}.

\begin{proof}[Proof of Lemma~\ref{lem:u.hat}]
We start with~\eqref{ineq:u-uhat} and set $\ubar=\uhat-u$, $\etabar=\etahat-\eta$. Observe that $(\ubar(t),\etabar(t))$ satisfies the following deterministic equation with random coefficients
\begin{align*}
\tfrac{\d}{\d t} \ubar(t)&=-A \ubar(t) +\int_{0}^{\infty}\close K(s) A\etabar(t;s)\d s+\f(\uhat(t))-\f(u(t))-K_1\alpha_{n_*}P_{n_*}\ubar(t),\\
\tfrac{\d}{\d t} \etabar(t)&= -\partial_s\etabar(t), \quad (\ubar(0),\etabar(0))=\x\in \H,\quad \etabar(t;0)=\ubar(t),\, t>0.
\end{align*}
It follows that (recalling $g$ as in~\eqref{form:g})
\begin{align*}
\frac{\d }{\d t}g(\ubar(t),\etabar(t))&=-\|\ubar(t)\|^2_{H^1}+ \int_0^\infty\close  K(s)\la \etabar(t;s),\ubar(t)\ra_{H^1}\d s+\la -\partial_s\etabar(t),\etabar(t)\ra_\M\\
&\qquad +\la\f(\uhat(t))-\f(u(t)),\ubar(t)\ra_H -K_1\alpha_{n_*}\|P_{n_*}\ubar(t)\|^2_H.
\end{align*}
Similarly to~\eqref{ineq:u+int.K<eta,u>:H1}, we readily have
\begin{align*}
-\|\ubar(t)\|^2_{H^1}&+ \int_0^\infty\close  K(s)\la \etabar(t;s),\ubar(t)\ra_{H^1}\d s+\la -\partial_s\etabar(t),\etabar(t)\ra_\M \\
 &\le -K_1\| \ubar(t)\|^2_{H^1}-\tfrac{\varepsilon_1\delta}{2(1+\varepsilon_1)}\|\etabar(t)\|^2_\M \\
 &\le -K_1\|(I-P_{n_*}) \ubar(t)\|^2_{H^1}-\tfrac{\varepsilon_1\delta}{2(1+\varepsilon_1)}\|\etabar(t)\|^2_\M \\
 &\le -K_1\alpha_{n_*}\|(I-P_{n_*}) \ubar(t)\|^2_{H}-\tfrac{\varepsilon_1\delta}{2(1+\varepsilon_1)}\|\etabar(t)\|^2_\M.
\end{align*}
In light of condition~\eqref{cond:phi:3},
\begin{align*}
\la\f(\uhat(t))-\f(u(t)),\ubar(t)\ra_H \le a_\f\|\ubar(t)\|^2_H.
\end{align*}
As a consequence, we obtain the almost sure bound
\begin{align*}
\frac{\d }{\d t}g(\ubar(t),\etabar(t)) \le -(K_1\alpha_{n_*}-a_\f)\|\ubar(t)\|^2_H-\tfrac{\varepsilon_1\delta}{2(1+\varepsilon_1)}\|\etabar(t)\|^2_\M,
\end{align*}
which together with the choice of $\alpha_{n_*}$ as in~\eqref{cond:K_1.alpha_n*>a_phi} and $(\ubar(0),\etabar(0))=\x$ clearly implies~\eqref{ineq:u-uhat}.

With regard to~\eqref{ineq:uhat}, we employ an argument similarly to the proof of Lemma~\ref{lem:moment-bound:H1} (in the base case $n=1$). In particular, following~\eqref{ineq:d.E(u,eta):H1}, we have the bound
\begin{align*}
\tfrac{\d }{\d t}\E \|(\uhat(t),\etahat(t))\|^2_{\H^1}&+K_1\E\|A\uhat(t)\|^2_{H}+\tfrac{\varepsilon_1\delta}{1+\varepsilon_1}\E\|\etahat(t)\|^2_{\M^1}\\
&\le \tfrac{a_\f^2}{K_1} \E\|\uhat(t)\|^2_{H}+\sum_{k\ge 1}\lambda_k^2\alpha_k
 -K_1\alpha_{n_*}\E\la P_{n_*}(\uhat(t)-u(t)),\uhat(t)\ra_{H^1}.
\end{align*}
Concerning the cross term on the above right-hand side, we employ Cauchy-Schwarz inequality to see that
\begin{align*}
-\la P_{n_*}(\uhat(t)-u(t)),\uhat(t)\ra_{H^1}\le \la P_{n_*}u(t),\uhat(t)\ra_{H^1}
&=\la P_{n_*}u(t),A\uhat(t)\ra_{H}\\
&\le \tfrac{\alpha_{n_*}}{2}\|u(t)\|^2_H+\tfrac{1}{2\alpha_{n_*}}\|A\uhat(t)\|^2_H.
\end{align*}
It follows that
\begin{align*}
\tfrac{\d }{\d t}\E \|(\uhat(t),\etahat(t))\|^2_{\H^1}&+\tfrac{1}{2} K_1\E\|A\uhat(t)\|^2_{H}+\tfrac{\varepsilon_1\delta}{1+\varepsilon_1}\E\|\etahat(t)\|^2_{\M^1}\\
&\le \tfrac{a_\f^2}{K_1} \E\|\uhat(t)\|^2_{H}+\tfrac{K_1\alpha_{n_*}^2}{2}\E\|u(t)\|^2_H+\sum_{k\ge 1}\lambda_k^2\alpha_k.
\end{align*}
Recalling~\eqref{ineq:moment-bound:H} and~\eqref{ineq:u-uhat},
\begin{align*}
\tfrac{a_\f^2}{K_1} \E\|\uhat(t)\|^2_{H}+\tfrac{K_1\alpha_{n_*}^2}{2}\E\|u(t)\|^2_H&\le C\E\|u(t)\|^2_H+C\E\|\uhat(t)-u(t)\|^2_H\\
&\le C e^{-ct}\|\x\|^2_\H+C.
\end{align*}
Gronwall's inequality then implies
\begin{align*}
\E \|(\uhat(t),\etahat(t))\|^2_{\H^1}\le C(\|\x\|^2_\H+1).
\end{align*}
This establishes~\eqref{ineq:uhat}, thereby finishing the proof.
\end{proof}

\section{Geometric Ergodicity of \eqref{eqn:react-diff:K}} \label{sec:ergodicity}

As mentioned in Section~\ref{sec:intro:overview}, the proof of Theorem~\ref{thm:ergodicity} makes use of the \emph{generalized coupling} in \cite{butkovsky2020generalized,glatt2017unique,
hairer2011asymptotic,kulik2017ergodic, kulik2015generalized,
mattingly2002exponential}. The method is based on two important concepts: a suitable distance $d$ in $\H$ that is \emph{contracting} for the semigroup $P_t$ and a $d$-small set to which the Markov process $\Phi(t)$ returns often enough. Given these ingredients, one is able to conclude the existence and uniqueness of an invariant probability measure $\mu$. Moreover, the convergent rate toward $\mu$ can be quantified by the recurrence rate, i.e., how quickly the system returns to the $d$-small set \cite{butkovsky2020generalized,hairer2011asymptotic,
kulik2017ergodic}. In turn, this can be estimated via suitable Lyapunov functions. 

For the reader convenience, we recall the notions of \emph{contracting} distances, \emph{d-small} sets and Lyapunov functions \cite{butkovsky2020generalized,hairer2011asymptotic}.

\begin{definition} \label{def:contracting}
A distance-like function $d$ bounded by 1 is called \emph{contracting} for $P_t$ if there exists $\alpha<1$ such that for any $\x,\y\in \H$ with $d(\x,\y)<1$, it holds that
\begin{equation} \label{ineq:contracting:W_d<alpha.d}
\W_d(P_t(\x,\cdot),P_t(\y,\cdot))\le \alpha d(\x,\y).
\end{equation}
\end{definition}

\begin{definition} \label{def:d-small} A set $B\subset\H$ is called $d$-\emph{small} for $P_t$ if for some $\varepsilon=\varepsilon(B)>0$, 
\begin{equation} \label{ineq:d-small:W_d<1-epsilon}
\sup_{\x,\y\in B}\W_d(P_t(\x,\cdot),P_t(\y,\cdot))\le 1-\varepsilon.
\end{equation}

\end{definition}

\begin{definition}\label{def:Lyapunov}
A function $V:\H\to [0,\infty)$ is called a Lyapunov function for $P_t$ if

1. $V(\x)\to\infty$ as $\|\x\|_\H\to\infty$.

2. There exist positive constants $c,\,C$ independent of $\x$ and $t$ such that
\begin{align*}
P_tV(\x)+c\int_0^t P_sV(\x)\d s\le V(\x)+Ct,\quad t\ge 0.
\end{align*}

\end{definition}

In Lemma~\ref{lem:contracting-d-small} below, we assert the existence of a contracting distance $d$ and the corresponding $d$-small set. The proof of Lemma~\ref{lem:contracting-d-small} will be deferred to the end of this section.

\begin{lemma} \label{lem:contracting-d-small} Under the same hypothesis of Proposition~\ref{prop:well-posed}, suppose that Assumption~\ref{cond:ergodicity} holds. Then, for all $R$, there exist $t_*=t_*(R)>1$ and $N=N(R)>0$ such that for all $t\ge t_*$: 

\begin{enumerate}[noitemsep,topsep=0pt,wide=\parindent, label=\arabic*.,ref=\theassumption.\arabic*.]
\item \label{lem:contracting-d-small:contracting} The distance $d_N(\x,\y)=N\|\x-\y\|_\H\mi 1$ as in~\eqref{form:d_N} is contracting for $P_t$ in the sense of Definition~\ref{def:contracting}.

\item \label{lem:contracting-d-small:d-small} The set $B_R=\{\x:\|\x\|_\H\le R\}$ is $d_N$-small for $P_t$ in the sense of Definition~\ref{def:d-small}.
\end{enumerate}
\end{lemma}

Given the above result, we can now conclude Theorem~\ref{thm:ergodicity}. See also \cite{butkovsky2020generalized,hairer2011asymptotic,kulik2017ergodic}.
\begin{proof}[Proof of Theorem~\ref{thm:ergodicity}]
We first recall $g(u,\eta)=\frac{1}{2}\|u\|^2_{H}+\frac{1}{2}\|\eta\|^2_{\M}$ as in~\eqref{form:g}. By integrating both sides of~\eqref{ineq:d.E[g(t)]} with respect to time, we obtain
\begin{align*}
\E g(u(t),\eta(t))+\int_0^t K_1\alpha_1\E\| u(r)\|^2_{H}+\tfrac{\varepsilon_1\delta}{2(1+\varepsilon_1)}\E\|\eta(r)\|^2_\M\d r\le \|\x\|^2_{\H}+\Big(a_3|\domain|+\tfrac{1}{2}\sum_{k\ge 1}\lambda_k^2\Big)t.
\end{align*}
So that $g(u(t),\eta(t))$ plays the role of the required Lyapunov function in the sense of Definition~\ref{def:Lyapunov}. In light of \cite[Proposition 2.1]{butkovsky2020generalized}, we combine the above estimate with Lemma~\ref{lem:contracting-d-small} to conclude the unique existence of $\mu$ as well as the convergent rate~\eqref{ineq:ergodicity:1}. 
\end{proof}

We now turn to the proof of Lemma~\ref{lem:contracting-d-small}. Before diving into the details, it is illuminating to review the \emph{generalized coupling} argument in \cite{butkovsky2020generalized,hairer2011asymptotic}. In order to establish required bounds on the Wassertstein distance $\W_{d_N}$ between $P_t(\x,\cdot)$ and $P_t(\y,\cdot)$, it is sufficient to compare the two solutions $\Phi_\x$ and $\Phi_\y$ of~\eqref{eqn:react-diff:K} starting from $\x$ and $\y$, respectively. However, we will not do so directly. To help with the analysis, we will consider instead an auxiliary system, denoted by $\Phitilde_{\x,\y}$, obtained from~\eqref{eqn:react-diff:K} by shifting the Wiener process $w(t)$ in a suitably chosen direction. We note that the change of measures is valid thanks to the assumption that noise is sufficiently forced in enough many directions  in $H$, cf. Assumption~\ref{cond:ergodicity}. As it turns out, the choice of $\Phitilde_{\x,\y}$ allows us to deduce two crucial estimates: firstly, $\Phi_\x(t)$ and $\Phitilde_{\x,\y}(t)$ can be arbitrarily close to one another as $t$ tends to infinity. Secondly, $\Phitilde_{\x,\y}(t)$ and $\Phi_\y(t)$ can be efffectively controlled. Altogether, we are able to conclude Lemma~\ref{lem:contracting-d-small}.

Recall from Assumption~\ref{cond:ergodicity} that
\begin{align*}
\big[ 1-\|K\|_{L^1(\rbb^+)} \big]\alpha_{\nbar}>a_\f=\sup_{x\in\rbb}\f'(x),
\end{align*}
where $\alpha_{\nbar}$ is as in~\eqref{cond:phi:ergodicity}. We set
\begin{equation} \label{form:eps_2.K_2}
\varepsilon_2=\frac{(1-\|K\|_{L^1(\rbb^+)})\alpha_{\nbar}-a_\f}{\|K\|_{L^1(\rbb^+)}\alpha_{\nbar}},\quad\text{and}\quad K_2 =  1-(1+\tfrac{\varepsilon_2}{2})\|K\|_{L^1(\rbb^+)}.
\end{equation}
Observe that $\varepsilon_2$ and $K_2$ are both positive by virtue of~\eqref{cond:phi:ergodicity}. Moreover, 
\begin{align}\label{cond:K_2.alpha_nbar>a_phi}
K_2\alpha_{\nbar}=\tfrac{1}{2}\big( \big[1-\|K\|_{L^1(\rbb^+)}\big]\alpha_{\nbar}+a_\f\big) >a_\f.
\end{align}
Next, we introduce the following ``shifted" system
\begin{align}
\d \utilde(t)&=-A \utilde(t)\d t +\int_{0}^{\infty}\close K(s) A\etatilde(t;s)\d s\d t+\f(\utilde(t))\d t+Q\d w(t)\notag\\
&\qquad\qquad+K_2\alpha_{\nbar} P_{\nbar}(u_\x(t)-\utilde(t))\d t,\label{eqn:react-diff:K:u.tilde}\\
\d \etatilde(t)&= -\partial_s\etatilde(t)\d t,\quad (\utilde(0),\etatilde(0))=\y\in\H,\quad \etatilde(t;0)=\utilde(t),\,t>0,\notag
\end{align}
where $K_2$ is as in~\eqref{form:eps_2.K_2}, $P_{\nbar}$ is the projection on to $\{e_1,\dots,e_{\nbar}\}$ as in~\eqref{form:P_n.u}, and $u_\x(t)$ is the $u$-component in the solution $\Phi_\x(t)$ of the original equation~\eqref{eqn:react-diff:K}. Similarly to~\eqref{eqn:react-diff:K:u.hat}, we note that~\eqref{eqn:react-diff:K:u.tilde} only differs from~\eqref{eqn:react-diff:K} by the appearance of the term $K_2\alpha_{\nbar} P_{\nbar}(u(t)-\utilde(t))\d t$. For notational convenience, we denote $$\Phitilde_{\x,\y}(t)=(\utilde_{\x,\y}(t),\etatilde_{\x,\y}(t)),$$ the solution of~\eqref{eqn:react-diff:K:u.tilde} with initial condition $\y\in\H$. 

Three of the main ingredients in the generalized coupling argument are given in the following results to be
proved at the end of this section.
\begin{lemma} \label{lem:dissipative-bound} 
Under the same hypothesis of Proposition~\ref{prop:well-posed}, suppose that Assumption~\ref{cond:ergodicity} holds. Let $\Phi_\x$ and $\Phitilde_{\x,\y}$ respectively be the solutions of~\eqref{eqn:react-diff:K} and~\eqref{eqn:react-diff:K:u.tilde} with initial conditions $\x,\,\y\in\H$. Then, 
\begin{equation}\label{ineq:dissipative-bound}
\|\Phi_\x(t)-\Phitilde_{\x,\y}(t)\|_\H\le e^{-\zeta t} \|\x-\y\|_\H ,\quad t>0,
\end{equation}
for some positive $\zeta$ independent of $\x,\,\y$ and $t$.
\end{lemma}

\begin{lemma} \label{lem:error-in-law} Under the same hypothesis of Proposition~\ref{prop:well-posed}, suppose that Assumption~\ref{cond:ergodicity} holds. Let $\Phi_\x$ and $\Phitilde_{\x,\y}$ respectively be the solutions of~\eqref{eqn:react-diff:K} and~\eqref{eqn:react-diff:K:u.tilde} with initial conditions $\x,\,\y\in\H$. Then,
there exists a positive constant $\zeta_1$ independent of $t$, $\x$ and $\y$ such that
\begin{equation} \label{ineq:error-in-law:1}
\W_{\emph{TV}}(\emph{Law}(\Phitilde_{\x,\y}(t)),P_t(\y,\cdot))\le \zeta_1\|\x-\y\|.
\end{equation}
Furthermore, for all $R>0$, there exists $\varepsilon=\varepsilon(R)\in(0,1)$ such that
\begin{equation}\label{ineq:error-in-law:2}
\W_\emph{TV}(\emph{Law}(\Phitilde_{\x,\y}(t)),P_t(\y,\cdot))\le 1-\varepsilon, \quad
\x,\,\y\in B_R.
\end{equation}
\end{lemma}

\begin{lemma} \label{lem:W_(d_N)<W_TV}
For all probability measures $\mu_1$, $\mu_2$ in $\emph{Pr}(\H)$, and $N>0$, 
\begin{equation} \label{ineq:W_(d_N)<W_TV}
\W_{d_N}(\mu_1,\mu_2)<\W_{\emph{TV}}(\mu_1,\mu_2),
\end{equation}
where $\W_{d_N}$ is the Wasserstein distance associated with $d_N$ as in~\eqref{form:d_N}.
\end{lemma}

Assuming the above results, we are ready to conclude Lemma~\ref{lem:contracting-d-small} whose argument is based on \cite[Proof of Theorem 2.4]{butkovsky2020generalized}. See also \cite{hairer2011asymptotic}.

\begin{proof}[Proof of Lemma~\ref{lem:contracting-d-small}]
1. Suppose $\x,\y\in \H$ such that $d_N(\x,\y)<1$. Recalling $d_N$ as in~\eqref{form:d_N}, this implies that $$d_N(\x,\y)=N\|\x-\y\|_\H<1.$$
By triangle inequality, for all $t\ge 0$,
\begin{equation} \label{ineq:W_(d_N):triangle}
\begin{aligned}
\W_{d_N}(P_t(\x,\cdot),P_t(\y,\cdot))&\le \W_{d_N}(P_t(\x,\cdot),\Law(\Phitilde_{\x,\y}(t)))+\W_{d_N}(\Law(\Phitilde_{\x,\y}(t)),P_t(\y,\cdot))\\&=I_1+I_2.
\end{aligned}
\end{equation}
In view of~\eqref{ineq:error-in-law:1} and~\eqref{ineq:W_(d_N)<W_TV}, we readily have
\begin{align*}
I_2=\W_{d_N}(\Law(\Phitilde_{\x,\y}(t)),P_t(\y,\cdot))\le \zeta_1\|\x-\y\|_\H,
\end{align*}
where $\zeta_1$ is given by~\eqref{form:zeta_1} below. Concerning $I_1$, by the dual formula~\eqref{form:W_d:dual-Kantorovich}, it holds that
\begin{align}
I_1=\W_{d_N}(P_t(\x,\cdot),\Law(\Phitilde_{\x,\y}(t))) &\le \E \big[N\|\Phi_\x(t)-\Phitilde_{\x,\y}(t)\|_\H\mi 1\big]\nonumber\\
&\le N\E \|\Phi_\x(t)-\Phitilde_{\x,\y}(t)\|_\H \le Ne^{-\zeta t}\|\x-\y\|_\H.\label{ineq:W_(d_N):I_1}
\end{align}
In the last estimate above, we made use of~\eqref{ineq:dissipative-bound} with $\zeta$ as in~\eqref{form:zeta} below. Altogether, we deduce the bound for all $t\ge 1$
\begin{align*}
\W_{d_N}(P_t(\x,\cdot),P_t(\y,\cdot)) \le \big(e^{-\zeta t}+\tfrac{\zeta_1}{N}\big)N\|\x-\y\|_\H\le \big(e^{-\zeta }+\tfrac{\zeta_1}{N}\big)N\|\x-\y\|_\H=\big(e^{-\zeta }+\tfrac{\zeta_1}{N}\big)d_N(\x,\y).
\end{align*}
By choosing $N$ large enough such that 
\begin{equation}\label{cond:N}
\alpha:= e^{-\zeta}+\tfrac{\zeta_1}{N}<1,
\end{equation}
we obtain
\begin{align*}
\W_{d_N}(P_t(\x,\cdot),P_t(\y,\cdot))\le \alpha d_N(\x,\y),
\end{align*}
which establishes that $d_N$ is contracting for $P_t$ as claimed.

2. Similarly to part 1., for all $\x,\y\in B_R$, we invoke~\eqref{ineq:error-in-law:2} and~\eqref{ineq:W_(d_N)<W_TV} to see that
\begin{align*}
\W_{d_N}(\Law(\Phitilde_{\x,\y}(t)),P_t(\y,\cdot))\le 1-\varepsilon,
\end{align*}
where $\varepsilon=\varepsilon(R)$ does not depend on $t,\x,\y$. Together with~\eqref{ineq:W_(d_N):triangle}-\eqref{ineq:W_(d_N):I_1}, we infer the bound
\begin{align*}
\W_{d_N}(P_t(\x,\cdot),P_t(\y,\cdot))&\le Ne^{-\zeta t}\|\x-\y\|_\H+1-\varepsilon\le 1-\varepsilon+2RNe^{-\zeta t}.
\end{align*}
By choosing $t_*=t_*(R)$ large enough such that
\begin{equation} \label{cond:t_*}
2RNe^{-\zeta t}<\tfrac{\varepsilon}{2},
\end{equation}
we obtain 
\begin{align*}
\W_{d_N}(P_t(\x,\cdot),P_t(\y,\cdot)) \le 1-\tfrac{\varepsilon}{2},\quad t\ge t_*,\, \x,\y\in B_R,
\end{align*}
which establishes that $B_R$ is $d_N$-small for $P_t$. The proof is thus complete.
\end{proof}

We now turn to the auxiliary results in Lemmas~\ref{lem:dissipative-bound}--\ref{lem:error-in-law}--\ref{lem:W_(d_N)<W_TV}. To prove Lemma~\ref{lem:dissipative-bound}, we will mainly invoke the choice of the index $\nbar$ as in~\eqref{cond:phi:ergodicity} to derive the exponential estimate~\eqref{ineq:dissipative-bound}. In turn, we will combine~\eqref{ineq:dissipative-bound} with the fact that $Q$ is invertible in span$\{e_1,\dots,e_{\nbar}\}$, cf.~\eqref{cond:Q:ergodicity}, to conclude Lemma~\ref{lem:error-in-law}. Finally, the proof of Lemma~\ref{lem:W_(d_N)<W_TV} is quite standard relying on the fact that the discrete metric $\mathbf{1}\{\x\neq \y\}$ dominates $d_N(\x,\y)$.

We first provide the proof of Lemma~\ref{lem:dissipative-bound} whose argument is similarly to that of Lemma~\ref{lem:u.hat}.

\begin{proof}[Proof of Lemma~\ref{lem:dissipative-bound}] To simplify notation, we will omit the subscripts $\x,\,\y$ in the proof.

For a slightly abuse of notation, we set $\ubar=u-\utilde$ and $\etabar=\eta-\etatilde$, from~\eqref{eqn:react-diff:K:u.tilde} and~\eqref{eqn:react-diff:K}, and observe that 
\begin{equation} \label{eqn:react-diff:K:u.bar}
\begin{aligned}
\tfrac{\d}{\d t} \ubar(t)&=-A \ubar(t) +\int_{0}^{\infty}\close K(s) A\etabar(t;s)\d s+\f(u(t))-\f(\utilde(t))-K_2\alpha_{\nbar} P_{\nbar}\ubar(t),\\
\tfrac{\d}{\d t} \etabar(t)&= -\partial_s\etabar(t).
\end{aligned}
\end{equation}
Following the same argument as in the proofs of Lemma~\ref{lem:moment-bound:H} and~\eqref{ineq:u-uhat}, equation~\eqref{eqn:react-diff:K:u.bar} implies (recalling $g$ as in~\eqref{form:g})
\begin{align*}
\tfrac{\d }{\d t}g(\ubar(t),\etabar(t))&=-\|\ubar(t)\|^2_{H^1}+ \int_0^\infty\close  K(s)\la \etabar(t;s),\ubar(t)\ra_{H^1}\d s+\la -\partial_s\etabar(t),\etabar(t)\ra_\M\\
&\qquad +\la\f(u(t))-\f(\utilde(t)),\ubar(t)\ra_H -K_2\alpha_{\nbar}\|P_{\nbar}\ubar(t)\|^2_H.
\end{align*}
Recalling $\rho(r)=\int_t^\infty K(s)\d s$,
\begin{align*}
\int_0^\infty\close  K(s)\la \etabar(t;s), \ubar(t)\ra_{H^1}\d s&\le \tfrac{1+\varepsilon_2}{2}\|K\|_{L^1(\rbb^+)}\| \ubar(t)\|^2_{H^1}+\tfrac{1}{2(1+\varepsilon_2)}\int_0^\infty\close  K(s)\|\etabar(t;s)\|^2_{H^1}\d s\\
&= \tfrac{1+\varepsilon_2}{2}\|K\|_{L^1(\rbb^+)}\| \ubar(t)\|^2_{H^1}-\tfrac{1}{2(1+\varepsilon_2)}\int_0^\infty\close  \rho'(s)\|\eta(t;s)\|^2_{H^1}\d s,
\end{align*}
where $\varepsilon_2>0$ is given by~\eqref{form:eps_2.K_2}. Also, recalling~\eqref{eqn:<partial_s.eta,eta>},
\begin{align*}
\la -\partial_s\etabar(t),\etabar(t)\ra_\M &= \tfrac{1}{2}\|K\|_{L^1(\rbb)}\|\ubar(t)\|^2_{H^1}+\tfrac{1}{2}\int_0^\infty\close \rho'(s)\|\etabar(t;s)\|^2_{H^1}\d s.
\end{align*}
We then deduce that
\begin{align*}
 &-\|\ubar(t)\|^2_{H^1}+ \int_0^\infty\close  K(s)\la \etabar(t;s),\ubar(t)\ra_{H^1}\d s+\la -\partial_s\etabar,\etabar\ra_\M \nonumber\\
 &\le -\Big[1-\big(1+\tfrac{\varepsilon_2}{2}\big)\|K\|_{L^1(\rbb^+)}\Big]\| u(t)\|^2_{H^1}+\tfrac{\varepsilon_2}{2(1+\varepsilon_2)}\int_0^\infty\close  \rho'(s)\|\eta(t;s)\|^2_{H^1}\d s\nonumber\\
 &\le -\Big[1-\big(1+\tfrac{\varepsilon_2}{2}\big)\|K\|_{L^1(\rbb^+)}\Big]\| u(t)\|^2_{H^1}-\tfrac{\varepsilon_2\delta}{2(1+\varepsilon_2)}\int_0^\infty\close  \rho(s)\|\eta(t;s)\|^2_{H^1}\d s\nonumber\\
 &= -K_2\| u(t)\|^2_{H^1}-\tfrac{\varepsilon_2\delta}{2(1+\varepsilon_2)}\|\eta(t)\|^2_\M  \\
 &\le -K_2\|(I-P_{\nbar}) \ubar(t)\|^2_{H^1}-\tfrac{\varepsilon_2\delta}{2(1+\varepsilon_2)}\|\etabar(t)\|^2_\M \\
 &\le -K_2\alpha_{\nbar}\|(I-P_{\nbar}) \ubar(t)\|^2_{H}-\tfrac{\varepsilon_2\delta}{2(1+\varepsilon_2)}\|\etabar(t)\|^2_\M,
\end{align*}
whence 
\begin{align*}
&-\|\ubar(t)\|^2_{H^1}+ \int_0^\infty\close  K(s)\la \etabar(t;s),\ubar(t)\ra_{H^1}\d s+\la -\partial_s\etabar,\etabar\ra_\M-K_2\alpha_{\nbar}\|P_{\nbar}\ubar(t)\|^2_H\\
&\le -K_2\alpha_{\nbar}\|\ubar(t)\|^2_{H}-\tfrac{\varepsilon_2\delta}{2(1+\varepsilon_2)}\|\etabar(t)\|^2_\M.
\end{align*}
To control the nonlinear term, we invoke condition~\eqref{cond:phi:ergodicity} to infer
\begin{align*}
\la\f(u(t))-\f(\utilde(t)),\ubar(t)\ra_H \le a_\f\|\ubar(t)\|^2_H.
\end{align*}
Altogether, we obtain
\begin{align} \label{ineq:ubar+etabar}
\tfrac{1}{2}\tfrac{\d}{\d t}\big(\|\ubar(t)\|^2_H+\|\etabar(t)\|^2_\M\big)&\le -(K_2\alpha_{\nbar}-a_\f)\|\ubar(t)\|^2_H-\tfrac{\varepsilon_2}{2(1+\varepsilon_2)}\delta  \|\etabar(t)\|^2_\M,
\end{align}
where $K_2$ and $\varepsilon_2$ are as in~\eqref{form:eps_2.K_2}. Thanks to~\eqref{cond:phi:ergodicity}, the choice of $K_2$ satisfies~\eqref{cond:K_2.alpha_nbar>a_phi}, namely, $$K_2\alpha_{\nbar}>a_\f.$$ We now set
\begin{equation} \label{form:zeta}
\zeta= \min\{2(K_2\alpha_{\nbar}-a_\f),\tfrac{\varepsilon_2}{1+\varepsilon_2}\delta\},
\end{equation}
which is positive. Estimate~\eqref{ineq:dissipative-bound} now follows from~\eqref{ineq:ubar+etabar}-\eqref{form:zeta}.

\end{proof}

Next, we present the proof of Lemma~\ref{lem:error-in-law}. 

\begin{proof}[Proof of Lemma~\ref{lem:error-in-law}] 
In order to prove~\eqref{ineq:error-in-law:1}-\eqref{ineq:error-in-law:2}, we first consider the following cylindrical Wiener process
\begin{equation} \label{form:w.tilde}
\d\wtilde(t)=\beta(t)\d t+\d w(t),
\end{equation}
where
\begin{equation} \label{form:beta(t)}
\beta(t)= K_2\alpha_{\nbar} Q^{-1}P_{\nbar}(u(t)-\utilde(t)),
\end{equation}
and $u(t)$ and $\utilde(t)$ are the first components of $\Phi_\x(t)$ and $\Phitilde_{\x,\y}(t)$, respectively. Since $Q$ is invertible on $H_{\nbar}=\text{span}\{e_1,\dots,e_{\nbar}\}$ by virtue of condition~\eqref{cond:Q:ergodicity}, we note that
\begin{align*}
\|\beta(t)\|^2_H=\|K_2\alpha_{\nbar} Q^{-1}P_{\nbar}(u(t)-\utilde(t))\|_H^2&\le (K_2\alpha_{\nbar}a_Q^{-1})^2\|P_{\nbar}(u(t)-\utilde(t))\|_H^2\\
&\le (K_2\alpha_{\nbar}a_Q^{-1})^2 e^{-2\zeta t}\|\x-\y\|^2_\H.
\end{align*}
In the last estimate above, we employed~\eqref{ineq:dissipative-bound} with $\zeta$ given by~\eqref{form:zeta}. As a consequence,
\begin{align} \label{ineq:int.beta}
\E\int_0^\infty\close \|\beta(r)\|^2_H\d r
& =\E \int_0^\infty\close\|K_2\alpha_{\nbar} Q^{-1}P_{\nbar}(u(r)-\utilde(r))\|^2_H\d r \le \zeta_1^2\|\x-\y\|^2_\H,
\end{align}
where
\begin{equation} \label{form:zeta_1}
\zeta_1=\frac{K_2\alpha_{\nbar}}{a_Q\sqrt{2\zeta}}.
\end{equation}
In light of \cite[Inequality (A.1) and Theorem A.2]{butkovsky2020generalized} together with~\eqref{ineq:int.beta}, it holds that
\begin{align} \label{ineq:error-in-law:1:w}
\W_{\TV}(\Law(\wtilde_{[0,t]}),\Law(w_{[0,t]}))&\le\tfrac{1}{2}\Big(\E\int_0^\infty\close \|\beta(r)\|^2_H\d r\Big)^{1/2} \le \tfrac{1}{2}\zeta_1\|\x-\y\|_\H.
\end{align}
On the other hand, we invoke \cite[Inequality (A.2)]{butkovsky2020generalized} to see that for all $\x,\y\in B_R$
\begin{align} 
\W_{\TV}(\Law(\wtilde_{[0,t]}),\Law(w_{[0,t]}))&\le 1-\tfrac{1}{2}e^{-\frac{1}{2}\E\int_0^\infty \|\beta(r)\|^2_H\d r}\notag \\
&\le 1-\tfrac{1}{2}e^{-\frac{1}{2}\zeta_1^2\|\x-\y\|^2_\H}\notag\\
&\le 1-\tfrac{1}{2}e^{-2\zeta_1^2R^2}.\label{ineq:error-in-law:2:w}
\end{align}
Now, observe that~\eqref{ineq:error-in-law:1:w}-\eqref{ineq:error-in-law:2:w} imply~\eqref{ineq:error-in-law:1}-\eqref{ineq:error-in-law:2} if we can show that
\begin{equation} \label{ineq:Law.P_t<Law.w}
\W_{\TV}(\Law(\Phitilde_{\x,\y}(t)),P_t(\y,\cdot))\le\W_{\TV}(\Law(\wtilde_{[0,t]}),\Law(w_{[0,t]})).
\end{equation} 
To see this, consider any coupling $(\wtilde^1,w^1)$ for $(\wtilde,w)$ and denote by $\Xtil_{\x,\y}(t)$, $X_{\y}(t)$ respectively the solutions of~\eqref{eqn:react-diff:K:u.tilde} and~\eqref{eqn:react-diff:K} associated with $\wtilde^1$ and $w^1$. It is clear that $(\Xtil_{\x,\y}(t),X_{\y}(t))$ is a coupling for $(\Phitilde_{\x,\y}(t),\Phi_\y(t))$. We note that by the uniqueness of weak solution, 
\begin{align*}
\{\wtilde^1(r)=w^1(r),\, r\in[0,t]\}\subseteq\{\Xtil_{\x,\y}(t)=X_{\y}(t)\}.
\end{align*}
It follows that if 
\begin{align*}
\boldsymbol{1}\{\wtilde^1_{[0,t]}\neq w^1_{[0,t]}\}=0,
\end{align*}
then
\begin{align*}
\boldsymbol{1}\{\Xtil_{\x,\y}(t)\neq X_{\y}(t)\}=0.
\end{align*}
In particular, 
\begin{align*}
\boldsymbol{1}\{\wtilde^1[0,t]\neq w^1_{[0,t]}\}\ge \boldsymbol{1}\{\Xtil_{\x,\y}(t)\neq X_{\y}(t)\}.
\end{align*}
By the dual formula~\eqref{form:W_d:dual-Kantorovich}, we establish~\eqref{ineq:Law.P_t<Law.w}, thereby concluding the proof.
\end{proof}

Lastly, we provide the proof of Lemma~\ref{lem:W_(d_N)<W_TV}, which will finally conclude the proof of Lemma~\ref{lem:contracting-d-small}.
\begin{proof}[Proof of Lemma~\ref{lem:W_(d_N)<W_TV}]
Let $(X,Y)$ be any bivariate random variable such that $X\sim\mu_1$ and $Y\sim\mu_2$. Recalling $d_N(X,Y)=N\|X-Y\|_\H\mi 1$ as in~\eqref{form:d_N}, by the formula~\eqref{form:W_d:dual-Kantorovich},
\begin{align*}
\W_{d_N}(\mu_1,\mu_2)\le \E\, d_N(X,Y)& =\E\big[d_N(X,Y)\mathbf{1}\{X\neq Y\}\big]+\E\big[d_N(X,Y)\mathbf{1}\{X= Y\}\big]\\
& =\E\big[d_N(X,Y)\mathbf{1}\{X\neq Y\}\big]\\
&\le \E\big[\mathbf{1}\{X\neq Y\}\big].
\end{align*}
Since the last implication above holds for any such pair $(X,Y)$, we invoke~\eqref{form:W_d:dual-Kantorovich} again to deduce
\begin{align*}
\W_{d_N}(\mu_1,\mu_2) \le \W_{\TV}(\mu_1,\mu_2),
\end{align*}
thereby finishing the proof.
\end{proof}

\section*{Acknowledgment}
The author thanks Nathan Glatt-Holtz and Vincent Martinez for fruitful discussions on the topic of this paper. The author also would like to thank the anonymous reviewer for their valuable comments and suggestions.

\appendix

\section{Well-posedness of \eqref{eqn:react-diff:K}} \label{sec:well-posed}

In this section, we discuss Proposition~\ref{prop:well-posed} whose proof relies on the construction of the weak solutions for~\eqref{eqn:react-diff:K}. We start with the Galerkin finite-dimensional approximation.
\subsection{Finite-dimensional approximation}
Recalling the projection $P_n$ onto the first $n$ wavenumbers as in~\eqref{form:P_n.u}, we set
\begin{align*}
u^n(t)=\sum_{k=1}^n \la u^n(t),e_k\ra_He_k,\qquad \text{and}\quad \eta^n(t;s)= \sum_{k=1}^n \la \eta^n(t;s), e_k\ra_H e_k.
\end{align*}
We then consider the pair $(u^n(t),\eta^n(t))$ solving the following finite-dimensional system
\begin{equation} \label{eqn:react-diff:K:Galerkin}
\begin{aligned}
\d u^n(t)&=-A u^n(t)\d t +\int_{0}^{\infty}\close K(s) A\eta^n(t;s)\d s\d t+P_n\f(u^n(t))\d t+P_nQ\d w(t),\\
\d \eta^n(t)&= -\partial_s\eta^n(t)\d t, \quad u^n(0)=P_n u_0\in H,\eta^n(0;\cdot)=P_n\eta_0(\cdot)\in M,\, \eta^n(t;0)=u^n(t).
\end{aligned}
\end{equation}
By Ito's formula, we have
\begin{align*}
\tfrac{1}{2}\d \big( \|u^n(t),\eta^n(t)\|^2_{\H}\big)&=-\|\grad u^n(t)\|^2_H\d t+ \int_0^\infty \close K(s)\la \grad \eta^n(t;s),\grad u^n(t)\ra_H\d s\d t\\
&\qquad+\la P_n\f( u^n(t) ),u^n(t)\ra_H\d t+\la u^n(t),P_nQ\d w(t)\ra_H\d t+\tfrac{1}{2}\sum_{k=1}^n \lambda_k^2\d t\\
&\qquad-\tfrac{1}{2}\int_0^\infty \close\rho(s) \partial_s\|\grad\eta^n(t;s)\|^2_H\d s\d t . 
\end{align*}
Similarly to the a priori bounds in Section~\ref{sec:moment-bound:solution}, e.g., the proof of Lemma~\ref{lem:moment-bound:H}, we proceed to estimate the above right-hand side as follows:

In view of~\eqref{ineq:u+int.K<eta,u>:H1}, the drift terms can be bounded by
\begin{align*}
-\|\grad u^n(t)\|^2_H\d t&+ \int_0^\infty \close K(s)\la \grad \eta^n(t;s),\grad u^n(t)\ra_H\d s\d t-\tfrac{1}{2}\int_0^\infty\close \rho(s) \partial_s\|\grad\eta^n(t;s)\|^2_H\d s\d t\\
&\le -K_1\|\grad u^n(t)\|^2_H\d t-\tfrac{\varepsilon_1\delta}{2(1+\varepsilon_1)}\|\eta^n(t)\|^2_\M\d t,
\end{align*}
where we recall $\varepsilon_1$ and $K_1=1-(1+\tfrac{\varepsilon_1}{2})\|K\|_{L^1(\rbb^+)} $ as in~\eqref{form:eps_1.and.K_1}. 

Concerning the non-linear term, we invoke~\eqref{cond:phi:2} to see that
\begin{align*}
\la P_n\f( u^n(t) ),u^n(t)\ra_H&=\la \f( u^n(t) ),P_nu^n(t)\ra_H= \la \f( u^n(t) ),u^n(t)\ra_H\le -a_2\|u^n(t)\|^{p+1}_{L^{p+1}}+a_3|\domain|,
\end{align*}
where $p$ is the exponent constant from Assumption~\ref{cond:phi:well-possed}.

Using Burkholder's inequality, the Martingale term can be bounded in expectation by
\begin{align*}
\E\sup_{0\le r\le t}\Big|\int_0^r\la u^n(\ell),P_nQ\d w(\ell)\ra_H\Big|^2\le \|Q\|^2_{L(H)}\E\int_0^t \|u^n(r)\|^2_H\d r.
\end{align*}

Altogether, we arrive at the following estimate
\begin{equation} \label{ineq:E|u^n,eta^n|}
\begin{aligned}
\E\|(u^n(t),\eta^n(t))\|^2_\H&+\int K_1\|\grad u^n(r)\|^2_H+\tfrac{\varepsilon_1}{2(1+\varepsilon_1)}\delta\E\|\eta^n(r)\|^2_\M+a_2\E\|u^n(r)\|^{p+1}_{L^{p+1}}\d r\\
&\le \|(u_0,\eta_0)\|^2_\H+ \Big(\tfrac{1}{2}\sum_{k\ge 1}\lambda_k^2+a_3|\domain|\Big)t.
\end{aligned}
\end{equation}
In particular, this implies the existence and uniqueness of the strong solution $(u^n(\cdot),\eta^n(\cdot))$ for~\eqref{eqn:react-diff:K:Galerkin} \cite{khasminskii2011stochastic,meyn2012markov,oksendal2003stochastic}. Moreover, combining Burkholder's and Gronwall's inequalities, we deduce the following bound in sup norm
and\begin{align*}
\E\sup_{0\le r\le t}\|(u^n(r),\eta^n(r))\|^2_\H\le \|(u_0,\eta_0)\|^2_\H e^{ct},
\end{align*}
for some positive constant $c$ independent of $t,\,n$ and initial condition $(u_0,\eta_0)$. Also, setting $q=(p+1)/p$,~\eqref{cond:phi:1} combined with~\eqref{ineq:E|u^n,eta^n|} implies that
\begin{align} \label{ineq:phi(u^n)}
\E\int_0^t\| \f(u^n(r))\|^q_{L^q}\d r\le c\,\E\int_0^t 1+\|u^n(r)\|^{p+1}\d r\le c(u_0,\eta_0,t). 
\end{align}
Furthermore, since $\eta^n$ solves the transport equation in $\rbb^n$
\begin{align*}
\partial\eta^n(t;s)=-\partial_s\eta^n(t;s),\quad \eta^n(t;0)=u^n(t),\quad \eta^n(0;s)=P^n\eta_0(s),
\end{align*}
$\eta^n$ admits the following representation \cite{bonaccorsi2012asymptotic}
\begin{equation} \label{form:eta^n}
\eta^n(t;s)=u^n(t-s)\boldsymbol{1}\{s\le t\}+P_n\eta_0(s-t)\boldsymbol{1}\{s> t\}.
\end{equation}

\subsection{Passage to the limit} As a consequence of the preceding subsection, we deduce the following limits (up to a subsequence)
\begin{align*}
u^n &\rightharpoonup^* u \text{ in } L^2(\Omega;L^\infty(0,T;H) ),\\
u^n &\rightharpoonup u \text{ in } L^2(\Omega; L^2(0,T;H^1 )),\\
u^n &\rightharpoonup u\text{ in } L^{p+1}(\Omega; L^{p+1}(0,T;L^{p+1})),\\
\f(u^n)&\rightharpoonup \chi  \text{ in }  L^{q}(\Omega; L^{q}(0,T;L^{q})),\\
\eta^n &\rightharpoonup^* \eta \text{ in } L^2(\Omega;L^\infty(0,T;\M)),\\
\eta^n &\rightharpoonup \eta \text{ in } L^2(\Omega; L^2(0,T;\M )).
\end{align*}
Furthermore, (see \cite[pg. 224]{robinson2001infinite})
 $$P_n\f(u^n)\rightharpoonup \chi \text{ in } L^{q}(\Omega; L^{q}(0,T;L^{q})).$$
Next, we proceed to prove that a.s. $\eta$ satisfies~\eqref{eqn:weak-solution:eta}, i.e., 
\begin{align} \label{form:eta:2} 
\eta(t;s)=u(t-s)\boldsymbol{1}\{s\le t\}+\eta_0(s-t)\boldsymbol{1}\{s> t\}.
\end{align}
To see this, consider any arbitrary $\etahat\in \M$. We first note that 
\begin{align*}
\int_0^r \rho(s)\la \grad u^n(r-s),\grad\etahat(s)\ra_H\d s =  \int_0^r \la \grad u^n(r-s),\grad(\rho(s)\etahat(s))\ra_H\d s,
\end{align*}
which converges to
\begin{align*}
\int_0^r \la \grad u(r-s),\grad(\rho(s)\etahat(s))\ra_H\d s
&=\int_0^r \rho(s)\la \grad u(r-s),\grad\etahat(s)\ra_H\d s,
\end{align*}
as $n$ tends to infinity, since $u^n \rightharpoonup u$ in $ L^2(0,T;H^1 )$. Also, for each $r\in[0,t]$, Cauchy-Schwarz inequality and the fact that $\rho$ is decreasing on $[0,\infty)$ yield the bound
\begin{align*}
\Big|\int_0^r \rho(s)\la \grad u^n(r-s),\grad\etahat(s)\ra_H\d s\Big|&\le \tfrac{1}{2}\rho(0)\int_0^t\|\grad u^n(s)\|^2_H\d s+\tfrac{1}{2}\int_0^\infty\close\rho(s)\|\grad \etahat(s)\|^2_H\d s\\
&= c+\tfrac{1}{2}\|\etahat\|^2_\M,
\end{align*}
which is a.s. integrable on $[0,t]$. The Dominated Convergence Theorem then implies a.s.
\begin{align} \label{lim:u^n->u}
\int_0^t \int_0^r \rho(s)\la \grad u^n(r-s),\grad\etahat(s)\ra_H\d s\d r\to\int_0^t \int_0^r \rho(s)\la \grad u(r-s),\grad\etahat(s)\ra_H\d s\d r,
\end{align}
as $n\to\infty$. On the other hand, since $P_n\eta_0$ converges to $\eta_0$ in $\M$, for all $r\in[0,t]$, it holds that
\begin{align*}
&\Big|\int_r^\infty\close  \rho(s)\la P_n\eta_0(s-r)-\eta_0(s-r),\etahat(s)\ra_{H^1}\d s\Big|^2\\
&\le \int_r^\infty\close \rho(s)\| P_n\eta_0(s-r)-\eta_0(s-r)\|^2_{H^1}\d s\int_0^\infty \close \rho(s)\|\etahat(s)\|^2_{H^1}\d s\\
&= \int_0^\infty\close \rho(s+r)\| P_n\eta_0(s)-\eta_0(s)\|^2_{H^1}\d s\, \|\etahat(r)\|^2_{\M}\\
&\le \| P_n\eta_0-\eta_0\|^2_{\M}\|\etahat\|^2_{\M}\to0,\quad n\to\infty.
\end{align*}
Using the Dominated Convergence Theorem again, we obtain a.s.
\begin{align} \label{lim:P_n.eta_0->eta_0}
\int_0^t \int_r^\infty\close  \rho(s)\la P_n\eta_0(s-r)-\eta_0(s-r),\etahat(s)\ra_{H^1}\d s\to 0,\quad n\to\infty.
\end{align}
We now combine~\eqref{lim:u^n->u} and~\eqref{lim:P_n.eta_0->eta_0} to deduce 
\begin{align*}
\eta^n(t)=u^n(t-s)\boldsymbol{1}\{s\le t\}+P_n\eta_0(s-t)\boldsymbol{1}\{s> t\}\rightharpoonup u(t-s)\boldsymbol{1}\{s\le t\}+\eta_0(s-t)\boldsymbol{1}\{s> t\},
\end{align*}
in $L^2(0,T;\M)$. This proves~\eqref{eqn:weak-solution:eta} by the uniqueness of weak limit. 

Next, we turn to establish~\eqref{eqn:weak-solution:eta:1}. Considering $\etatilde\in \M$ such that $\partial_s\etatilde\in\M$, we multiply both sides of $\eta^n$-equation in~\eqref{eqn:react-diff:K:Galerkin} with $\etatilde$ and perform integration by parts to obtain
\begin{align*}
\tfrac{\d}{\d t}\la \eta^n(t),\etatilde\ra_\M&=\la -\partial_s\eta^n(t),\etatilde\ra_\M\\
&= -\rho(s)\la \eta^n(t;s),\etatilde(s)\ra_{H^1}\Big|^\infty_{0}+\int_0^\infty\close  \la \eta^n(t;s),\partial_s(\rho(s)\etatilde(s))\ra_{H^1}\d s\\
&= \rho(0)\la \eta^n(t;0),\etatilde(0)\ra_{H^1}+\la \eta^n(t),\partial_s\etatilde\ra_\M+\int_0^\infty\close\rho'(s)\la \eta^n(r;s),\etatilde(s)\ra_{H^1}\d s\d r\\
&= \rho(0)\la u^n(t),\etatilde(0)\ra_{H^1}+\la \eta^n(t),\partial_s\etatilde\ra_\M+\int_0^\infty\close\rho'(s)\la \eta^n(r;s),\etatilde(s)\ra_{H^1}\d s\d r.
\end{align*}
In the last implication above, we employed the fact that $\eta^n(t;0)=u^n(t)$. Integrating the above equation with respect to time $t$ yields
\begin{align*}
\la \eta^n(t),\etatilde\ra_\M &=\la P_n\eta_0,\etatilde\ra_\M+\int_0^t\rho(0)\la u^n(r),\etatilde(0)\ra_{H^1}\d r + \int_0^t\la \eta^n(r),\partial_s\etatilde\ra_{\M}\d r\\
&\qquad\qquad +\int_0^t\int_0^\infty\close\rho'(s)\la \eta^n(r;s),\etatilde(s)\ra_{H^1}\d s\d r.
\end{align*}
By sending $n$ to infinity, we deduce the identity~\eqref{eqn:weak-solution:eta:1} provided that
\begin{align} \label{lim:eta^n->eta}
\int_0^t\int_0^\infty\close\rho'(s)\la \eta^n(r;s),\etatilde(s)\ra_{H^1}\d s\d r\to \int_0^t\int_0^\infty\close\rho'(s)\la \eta(r;s),\etatilde(s)\ra_{H^1}\d s\d r,\quad n\to\infty.
\end{align}
To see this, recall from~\eqref{ineq:K/rho<c} that $$\frac{K(s)}{\rho(s)}=\frac{|\rho'(s)|}{\rho(s)}<c,\quad s\ge 0.$$
Thus, $\frac{|\rho'|}{\rho}\etatilde\in\M$, whence
\begin{align*}
\int_0^t\int_0^\infty\close\rho'(s)\la \eta^n(r;s),\etatilde(s)\ra_{H^1}\d s\d r&= \int_0^t\int_0^\infty\close\rho(s)\la \eta^n(r;s),\tfrac{\rho'(s)}{\rho(s)}\etatilde(s)\ra_{H^1}\d s\\
&\to \int_0^t\int_0^\infty\close\rho(s)\la \eta(r;s),\tfrac{\rho'(s)}{\rho(s)}\etatilde(s)\ra_{H^1}\d s\d r,
\end{align*}
as $n\to\infty$. This proves~\eqref{lim:eta^n->eta}, which in turn implies~\eqref{eqn:weak-solution:eta:1}.

We are left to establish~\eqref{eqn:weak-solution:u}. Considering any $v\in H^1\cap L^{p+1}$, it holds that
\begin{equation} \label{eqn:<u^n,v>}
\begin{aligned}
\la u^n(t),v\ra_H&= \la P_nu_0,v\ra_H-\int_0^t \la u^n(r),v\ra_{H^1}+\int_0^t \int_0^\infty\close  K(s)\la \eta^n(r;s),v\ra_{H^1}\d s\d r\\
&\qquad+\int_0^t \la P^n\f(u^n(r)),v\ra_H\d r+\int_0^t \la v,P_nQ\d w(r)\ra_H.   
\end{aligned}
\end{equation}
We first claim that
\begin{align*}
\int_0^t \int_0^\infty \close K(s)\la \eta^n(r;s),v\ra_{H^1}\d s\d r\to \int_0^t \int_0^\infty\close  K(s)\la \eta(r;s),v\ra_{H^1}\d s\d r,\quad n\to\infty.
\end{align*}
To this end, recalling~\eqref{ineq:K/rho<c} again, we see that for all $s\ge 0$, $\frac{K(s)}{\rho(s)}<c$, implying
\begin{align*}
\int_0^\infty \close \rho(s)\cdot \frac{K(s)^2}{\rho(s)^2} \|v\|^2_{H^1}\d s\le c^2\|v\|^2_{H^1}\|\rho\|_{L^1(\rbb^+)}. 
\end{align*}
In other words, $\frac{K(\cdot)}{\rho(\cdot)}v \in \M$. It follows that as $n\to\infty$,
\begin{align*}
\int_0^t \int_0^\infty\close  K(s)\la \eta^n(r;s),v\ra_{H^1}\d s\d r&=\int_0^t \int_0^\infty\close  \rho(s)\la \eta^n(r;s),\tfrac{K(s)}{\rho(s)} v\ra_{H^1}\d s\d r\\
&\to \int_0^t \int_0^\infty\close  \rho(s)\la \eta(r;s),\tfrac{K(s)}{\rho(s)} v\ra_{H^1}\d s\d r\\
&= \int_0^t \int_0^\infty\close  K(s)\la \eta(r;s),v\ra_{H^1}\d s\d r.
\end{align*}
As a consequence, sending $n\to\infty$ in \eqref{eqn:<u^n,v>} yields 
\begin{align*}
\la u(t),v\ra_H&= \la u_0,v\ra_H-\int_0^t \la u(r),v\ra_{H^1}+\int_0^t \int_0^\infty\close  K(s)\la \eta(r;s),v\ra_{H^1}\d s\d r\\
&\qquad+\int_0^t \la \chi(r),v\ra_H\d r+\int_0^t \la v,Q\d w(r)\ra_H.   
\end{align*}
It therefore remains to establish that $\chi=\f(u)$. The argument follows along the lines of \cite[Theorem 8.4]{robinson2001infinite} tailored to our setting (see also \cite{caraballo2008pullback,glatt2008stochastic}). 

We first consider the pair $(P_n u(t),P_n\eta(t))$ and note that they obey the finite-dimensional system
\begin{equation} \label{eqn:react-diff:mu:P_N.u,P_N.eta}
\begin{aligned}
\d P_n u(t)&=-AP_n u(t)\d t-\int_0^\infty\close  K(s)AP_n \eta(t;s)\d s\d t+P_n \chi(t)\d t+P_n Q\d w(t),\\
\d P_n  \eta(t)&= -\partial_s P_n \eta(t)\d t,\\
 P_n  u(0)&=P_n  u_0,\quad P_n  \eta(0)=P_n \eta_0,\quad P_n\eta(t;0)=P_nu(t),\,t>0.
\end{aligned}
\end{equation} 
Setting $Z_n=P_n u-u^n,\,\zeta_n=P_n \eta-\eta^n$ and subtracting~\eqref{eqn:react-diff:K:Galerkin} from~\eqref{eqn:react-diff:mu:P_N.u,P_N.eta}, observe that
\begin{equation} \label{eqn:react-diff:mu:ubar,etabar}
\begin{aligned}
\tfrac{\d}{\d t} Z_n(t)&=-AZ_n(t)-\int_0^\infty \close K(s)A\zeta_n(t;s)\d s+P_n \chi(t)-P_n (\f(u^n(t))),\\
\tfrac{\d}{\d t} \zeta_n(t)&= -\partial_s \zeta_n(t),\qquad Z_n(0)=0,\quad \zeta_n(0)=0,\quad\zeta_n(t;0)=Z_n(t),\,t>0.
\end{aligned}
\end{equation} 
It follows that
\begin{align*}
\tfrac{1}{2}\cdot\tfrac{\d}{\d t} \|(Z_n(t),\zeta_n(t))\|^2_{\H}
&=-\|\grad Z_n(t)\|^2_H+ \int_0^\infty \close K(s)\la \grad \zeta_n(t;s),\grad Z_n(t)\ra_H\d s-\tfrac{1}{2}\int_0^\infty \close\rho(s) \partial_s\|\grad\zeta_n(t;s)\|^2_H\\
&\qquad+\la P_n \chi(t)-P_n (\f(u^n(t))),Z_n(t)\ra_H.
\end{align*} 
Similarly to~\eqref{ineq:u+int.K<eta,u>:H1}, we readily have the bound
\begin{align*}
-\|\grad Z_n(t)\|^2_H+& \int_0^\infty \close K(s)\la \grad \zeta_n(t;s),\grad Z_n(t)\ra_H\d s-\tfrac{1}{2}\int_0^\infty \close\rho(s) \partial_s\|\grad\zeta_(t;s)\|^2_H\\
&\le -K_1\|\grad Z_n(t)\|^2_H-\tfrac{\varepsilon_1\delta}{2(1+\varepsilon_1)}\|\zeta_n(t)\|^2_\M,
\end{align*}
where we recall $\varepsilon_1$ and $K_1=1-(1+\tfrac{\varepsilon_1}{2})\|K\|_{L^1(\rbb^+)} $ as in~\eqref{form:eps_1.and.K_1}. 
Concerning the non-linear term involving $\chi$, we write
\begin{align*}
&\la P_n \chi(t)-P_n (\f(u^n(t))),P_n  u(t)-u^n(t)\ra_H\\
& = \la \chi(t)-\f(u^n(t)),P_n  u(t)-u^n(t)\ra_H\\
& = \la \chi(t),P_n  u(t)-u^n(t)\ra_H+ \la \f(u(t))-\f(u^n(t)), u(t)-u^n(t)\ra_H\\
&\qquad +\la (I-P_n )\f(u^n(t)),u(t)\ra_H- \la \f(u(t)),u(t)-u^n(t)\ra_H\\
&=I_1+I_2+I_3-I_4.
\end{align*}
Concerning $I_1$, since both $u^n$ and $P_n u$ converge weakly to $u$ in $ L^{p+1}(\Omega; L^{p+1}(0,T;L^{p+1}))$, we obtain the limit
$$\E \int_0^tI_1(r)\d r\to 0,\text{ as }n\to \infty.$$ 
Also, using Holder's inequality,
\begin{align*}
\Big|\E \int_0^tI_1(r)\d r\Big| \le \Big(\E\int_0^T\close \|\chi(r)\|^{q}_{L^q}\d r \Big)^{\frac{1}{q}}\Big(\E\int_0^T \close\|P_n  u(r)-u^n(r)\|_{L^{p+1}}^{p+1}\d r \Big)^{\frac{1}{p+1}}\le C,
\end{align*}
for some positive constant $C=C(T)$. To bound $I_4$, we note that since $u\in L^{p+1}(\Omega; L^{p+1}(0,T;L^{p+1}))$, condition~\eqref{cond:phi:1} implies $\f(u)\in L^{q}(\Omega; L^{q}(0,T;L^{q}))$. Similarly to $I_1$, we have
$$\E\int_0^t I_4(r)\d r\to 0,\text{ as }n\to \infty,$$
and
\begin{align*}
\Big|\E \int_0^tI_4(r)\d r\Big| \le \Big(\E\int_0^T\close \|\f(u(r))\|^{q}_{L^q}\d r \Big)^{\frac{1}{q}}\Big(\E\int_0^T \close\|u(r)-u^n(r)\|_{L^{p+1}}^{p+1}\d r \Big)^{\frac{1}{p+1}}\le C.
\end{align*}
With regard to $I_3$, since $\f(u^n)$ is uniformly bounded in $L^{q}(0,T;L^{q})$, cf.~\eqref{ineq:phi(u^n)}, it holds that
\begin{align*}
(I-P_n )\f(u^n)\rightharpoonup 0, \text{ in } L^{q}(0,T;L^{q}).
\end{align*}
In particular, we have
\begin{align*}
\E\int_0^t I_3(r)\d r\to 0,\text{ as }n\to \infty,
\end{align*}
and
\begin{align*}
\Big|\E \int_0^tI_3(r)\d r\Big| \le \Big(\E\int_0^T\close \|(I-P_n)\f(u^n(r))\|^{q}_{L^q}\d r \Big)^{\frac{1}{q}}\Big(\E\int_0^T \close\|u(r)\|_{L^{p+1}}^{p+1}\d r \Big)^{\frac{1}{p+1}}\le C.
\end{align*}
Concerning $I_2$, we invoke condition~\eqref{cond:phi:3} again to infer
\begin{align*}
I_2=\la \f(u(t))-\f(u^n(t)), u(t)-u^n(t)\ra_H&\le a_\f\|u(t)-u^n(t)\|^2_H\\
&=a_\f \|P_n  u(t)-u^n(t)\|^2_H+a_\f\|(I-P_n )u(t)\|^2_H.
\end{align*}
Setting
\begin{align*}
G_n(t):=\E\int_0^t I_1(r)+I_2(r)-I_4(r)+a_\f\|(I-P_n )u(r)\|^2_H\d r,
\end{align*}
by Gronwall's inequality, we infer the bound a.e. in $[0,T]$
\begin{align*}
\E\|(Z_n(t),\zeta_n(t))\|^2_\H\le |G_n(t)|+C\int_0^t |G_n(r)|\d r,
\end{align*}
for some positive constant $C=C(T)$. We observe that $G_n(t)$ converges to zero and that 
\begin{align*}
|G_n(t)|\le \Big|\E \int_0^tI_1(r)\d r\Big|+\Big|\E \int_0^tI_3(r)\d r\Big|+\Big|\E \int_0^t I_4(r)\d r\Big|+a_\f\E|\int_0^T\close\|u(r)\|^2_H\d r\le C.
\end{align*}
By the Dominated Convergence Theorem, it holds that
\begin{align*}
\int_0^t |G_n(r)|\d r\to 0,\text{ as }n\to \infty,
\end{align*}
whence for a.e. $t\in[0,T]$
\begin{align*}
\E\|(Z_n(t),\zeta_n(t))\|^2_\H\to 0,\text{ as }n\to \infty.
\end{align*}
We invoke the Dominated Convergence Theorem again to deduce
\begin{align*}
\E\int_0^T\close\|(Z_n(r),\zeta_n(r))\|^2_\H\d r\to 0,\text{ as }n\to \infty.
\end{align*}
In particular, this implies that (up to a subsequence)
\begin{align*}
u^n \to u \text{  a.s. in  }C(0,T;H),
\end{align*}
and thus, (up to a further subsequence) $u^n$ converges to $u$ a.e. $(x,t)\in \domain\times[0,T]$. It follows that $\f(u^n)$ converges to $\f(u)$ a.e. $(x,t)\in \domain\times [0,T]$ since $\f$ is continuous. In view of~\cite[Lemma 8.3]{robinson2001infinite}, we obtain
$$\f(u^n )\rightharpoonup \f(u)  \text{ in }   L^{q}(0,T;L^{q}),$$
whence a.s., $\chi=\f(u)$ a.e. $(x,t)\in \domain\times [0,T]$. This concludes the proof of Proposition~\ref{prop:well-posed}.

\bibliographystyle{abbrv}
{\footnotesize\bibliography{react-diff}}

\end{document}